\newif\ifGD
\GDfalse
\newif\ifcmts
\cmtstrue

\ifGD
\documentclass[runningheads,a4paper]{llncs}
\else
\documentclass[11pt]{article}
\fi
\usepackage{url}
\usepackage{graphicx}
\usepackage{authblk} 
\usepackage{color}
\usepackage{enumerate}
\ifGD
\usepackage{amssymb,amsmath}
\usepackage{wrapfig}
\let\doendproof\endproof
\renewcommand\endproof{~\hfill\qed\doendproof}
\else
\usepackage{amsmath,amssymb,amsthm}
\usepackage{fullpage}
\fi

\ifGD    
\newcommand{\keywords}[1]{\par\addvspace\baselineskip
\noindent\keywordname\enspace\ignorespaces#1}
\fi

\ifGD
\newtheorem{numclaim}[theorem]{Claim}
\else
\newtheorem{theorem}{Theorem}[section]
\newtheorem{lemma}[theorem]{Lemma}
\newtheorem{corollary}[theorem]{Corollary}
\newtheorem{proposition}[theorem]{Proposition}

\newtheorem{claim}[theorem]{Claim}
\newtheorem*{claim*}{Claim}
\theoremstyle{definition}

\theoremstyle{remark}
\newtheorem{example}[theorem]{Example}
\newtheorem*{example*}{Example}
\newtheorem{remark}[theorem]{Remark}
\fi

\newcommand\R{\ensuremath{\mathbb{R}}}
\newcommand\Z{\ensuremath{\mathbb{Z}}}

\newcommand{\heading}[1]{\vspace{1ex}\par\noindent{\bf\boldmath #1}}
\renewcommand\:{\colon}

\DeclareMathOperator{\supp}{supp}
\DeclareMathOperator{\im}{im}

\DeclareMathOperator{\smod}{mod} 

\newcommand{\cut}[2]{#1_{\langle #2\rangle}} 		
\newcommand{\sign}[2]{[#2:#1]}
\newcommand{\fno}{\hat{f}}

\newcommand\surf{{\mathcal{M}}}
\newcommand\surff{{\mathcal{P}}}

\newcommand\nsurf{{\mathcal{N}}}

\newcommand\lsys{{L}}

\newlength{\fparwidth}
\newcommand\framedpar[1]{\begin{center}\framebox{~\begin{minipage}
{\fparwidth}\vspace{1mm}#1\vspace{1mm}
\end{minipage}~}\end{center}}

\newcommand{\maxmn}{L}

\long\def\onefigure#1#2{
\begin{figure*}[tbp]
\begin{center}
#1
\end{center}
\caption{#2}
\end{figure*}
}

\def\immediateFigure#1{%
\smallskip\begin{center}#1\end{center}\smallskip }

\newcommand{\labfig}[2]  
{\onefigure{\mbox{\includegraphics{#1}}}{\label{f:#1} #2} }

\newcommand{\labfigw}[3]  
{\onefigure{\mbox{\includegraphics[width=#2]{#1}}}{\label{f:#1} #3}}

\newcommand{\labfigscale}[3]
{\onefigure{\mbox{\includegraphics[scale=#2]{#1}}}{\label{f:#1} #3}}

\newcommand{\immfig}[1]  
{\immediateFigure{\mbox{\includegraphics{#1}}}}

\newcommand{\immfigw}[2] 
{\immediateFigure{\mbox{\includegraphics[width=#2]{#1}}}}

\newcommand{\marrow}{\marginpar{\boldmath$\longleftarrow$}}
\definecolor{violet}{rgb}{0.5,0,0.5}
\newcommand{\jirka}[1]{\ifcmts\ifhmode\newline\fi\marrow \textsf{\color{magenta}*** (JIRKA: ) #1\newline}\else\fi}
\newcommand{\martin}[1]{\ifcmts\ifhmode\newline\fi\marrow \textsf{\color{violet}***
(Martin: ) #1\newline}\else\fi}
\newcommand{\uli}[1]{\ifcmts\ifhmode\newline\fi\marrow \textsf{\color{cyan}***
(ULI: ) #1\newline}\else\fi}

\ifGD
\else
\title{Untangling two systems of noncrossing curves}

\author[1,2,a,b]{Ji\v{r}\'{\i} Matou\v{s}ek}
\author[3,a]{Eric Sedgwick}
\author[1,4,a,c]{Martin Tancer}
\author[5,a,d]{Uli Wagner}
\affil[1]{Department
of Applied Mathematics, 
Charles University, Malostransk\'{e} n\'{a}m.
25, 118~00~~Praha~1, Czech Republic.}
\affil[2]
{Department of Computer Science, ETH Z\"urich, 8092 Z\"urich,
Switzerland}
\affil[3]{School of Computing, DePaul University, 243 S.~Wabash Ave, 
Chicago, IL 60604, USA}
\affil[4]{Institutionen f\"{o}r matematik, Kungliga Tekniska
H\"{o}gskolan, 100~44 Stockholm.}
\affil[5]{IST Austria, Am Campus 1, 3400 Klosterneuburg, Austria}
\affil[a]{Supported by the ERC Advanced Grant No.~267165.}
\affil[b]{Partially supported by Grant GRADR Eurogiga GIG/11/E023.}
\affil[c]{Supported by a G\"{o}ran Gustafsson
postdoctoral fellowship.}
\affil[d]{Supported by the Swiss National Science Foundation (Grant SNSF-PP00P2-138948).}
\fi
\begin{document}

\ifGD 

\mainmatter  

\title{Untangling two systems of noncrossing curves\thanks{Research supported by the ERC Advanced Grant No.~267165. Moreover, the research of J.M.\ was also partially supported by Grant GRADR Eurogiga GIG/11/E023, the research of M.T.\ was supported by a G\"{o}ran Gustafsson postdoctoral fellowship, and the research of U.W.\ was supported by the Swiss National Science Foundation (Grant SNSF-PP00P2-138948) }}

\titlerunning{Untangling two systems of noncrossing curves}

\author{Ji\v{r}\'{\i} Matou\v{s}ek
\inst{1}$^,$\inst{2} \and Eric Sedgwick\inst{3} \and Martin Tancer\inst{1}$^,$\inst{4} \and Uli Wagner\inst{5}}

\authorrunning{Matou\v{s}ek, Sedgwick, Tancer, Wagner}

\institute{
Department
of Applied Mathematics, 
Charles University, Malostransk\'{e} n\'{a}m.
25, 118~00~~Praha~1, Czech Republic
\and
Institute of Theoretical Computer Science, ETH Z\"urich, 8092 Z\"urich,
Switzerland
\and
School of Computing, DePaul University, 243 S.~Wabash Ave, 
Chicago, IL 60604, USA
\and Institutionen f\"{o}r matematik, KTH, 100~44 Stockholm, Sweden
\and
IST Austria, Am Campus 1, 3400 Klosterneuburg, Austria
}
\fi
\maketitle

\begin{abstract}
We consider two systems $(\alpha_1,\ldots,\alpha_m)$
and $(\beta_1,\ldots,\beta_n)$ of simple curves drawn on a compact two-dimensional surface $\surf$ with boundary. 
\ifGD\else

\fi
Each $\alpha_i$ and each $\beta_j$ is either an arc meeting the
boundary of $\surf$ at its two endpoints, or a closed curve.
The $\alpha_i$ are pairwise disjoint except for possibly
sharing endpoints, and similarly for the $\beta_j$. We want to ``untangle''
the $\beta_j$ from the $\alpha_i$ by a self-homeomorphism of $\surf$;
more precisely, we seek a homeomorphism $\varphi\:\surf\to\surf$ fixing the
boundary of $\surf$ pointwise such that the total number of crossings of the $\alpha_i$
with the $\varphi(\beta_j)$ is as small as possible. This problem is
motivated by an application in the algorithmic theory of embeddings and $3$-manifolds.

We prove that if $\surf$ is planar, i.e., a sphere with $h\geq 0$ boundary components (``holes''), 
then $O(mn)$ crossings can be achieved (independently of $h$), which is asymptotically
tight, as an easy lower bound shows.
\ifGD\else

\fi
In general, for an arbitrary (orientable or nonorientable) surface $\surf$ with $h$ holes and of 
(orientable or nonorientable) genus $g\geq 0$, we obtain an $O((m+n)^4)$ upper bound, 
again independent of $h$ and~$g$.
\ifGD\else

The proofs rely, among other things, on a result concerning simultaneous planar drawings of graphs by 
Erten and Kobourov.
\fi

\ifGD\keywords{Curves on $2$-manifolds, simultaneous planar drawings,
Lickorish's theorem}\fi
\end{abstract}

\section{Introduction}
\label{s:intro}

Let $\surf$ be a surface, by which we mean a two-dimensional
compact manifold with (possibly empty) boundary $\partial\surf$. (Unless stated
otherwise, we work with connected surfaces.)

By the classification theorem for surfaces, if $\surf$ is orientable, then 
$\surf$ is homeomorphic to a sphere with $h\ge 0$ \emph{holes} and $g\ge 0$ 
attached handles \ifGD\else (see Fig.~\ref{f:s32})\fi; 
the number $g$ is also called the \emph{orientable genus} of $\surf$.
If $\surf$ is nonorientable, then it is homeomorphic to a sphere with $h\geq 0$ holes and with 
$g\geq 0$ \emph{cross-caps};\ifGD\ \else\footnote{A cross-cap is obtained by removing a small disc from
$\surf$ and gluing in a M\"{o}bius band along its boundary to the boundary circle of the resulting 
hole.}
 \fi 
in this case, the integer $g$ is known as the \emph{nonorientable genus} of $\surf$. 
In the sequel, the word ``genus'' will mean orientable genus for orientable surfaces and nonorientable genus for nonorientable surfaces.

We will consider curves in $\surf$ that are \emph{properly embedded}, i.e., every 
curve is either a simple arc meeting the boundary $\partial\surf$ exactly at its two endpoints,
or a simple closed curve avoiding $\partial\surf$. An \emph{almost-disjoint 
system of curves} in $\surf$ is a collection $A=(\alpha_1,\ldots,\alpha_m)$ 
of curves that are pairwise disjoint except for possibly sharing
endpoints.\footnote{We use ordered collections of curves just because of the
convenience of the notation.}

In this paper we consider the following problem: We are given two 
almost-disjoint systems $A=(\alpha_1,\ldots,\alpha_m)$ and
$B=(\beta_1,\ldots,\beta_n)$ of curves in $\surf$, where the curves
of $B$ intersect those of $A$ possibly very many times,
as in Fig.~\ref{f:intro-ex}(a). We would like to ``redraw'' the curves of $B$
in such a way that they intersect those of $A$ as little as
possible.
\ifGD
\labfigscale{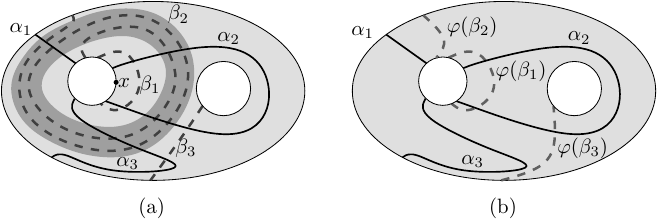}{1}
{Systems $A$ and $B$ of curves on a surface $\surf$, with $g=0$ and $h=3$ (a), and a re-drawing of $B$ via
a $\partial$-automorphism $\varphi$  reducing
the number of intersections (b).}
\else
\labfig{intro-ex}
{Systems $A$ and $B$ of curves on a surface $\surf$, with $g=0$ and $h=3$ (a), and a re-drawing of $B$ via
a $\partial$-automorphism $\varphi$ (composed of an isotopy and a \emph{Dehn twist} of the darkly shaded annular region, see below) 
so that the number of intersections is reduced (b).}
\fi

We consider re-drawings only in a restricted sense, namely, induced
by \emph{$\partial$-automorphisms} of $\surf$, where a $\partial$-automorphism
is a homeomorphism $\varphi\:\surf\to\surf$ that fixes 
the boundary $\partial\surf$ pointwise.\footnote{In general, by an automorphism
we mean a self-homeomorphism.} Thus, given the $\alpha_i$ and the $\beta_j$, 
we are looking for a $\partial$-automorphism $\varphi$ such that the number of 
intersections (crossings) between $\alpha_1,\ldots,\alpha_m$ and $\varphi(\beta_1),\ldots,
\varphi(\beta_n)$ is as small as possible (where sharing endpoints
does not count). We call this minimum number of crossings achievable through
any choice of $\varphi$ the \emph{entanglement number}
of the two systems $A$ and $B$.

In the orientable case, let $f_{g,h}(m,n)$ denote the maximum entanglement number
of any two systems $A=(\alpha_1,\ldots,\alpha_m)$ and $B=(\beta_1,\ldots,\beta_n)$
of almost-disjoint curves on an orientable surface of genus $g$ with $h$ holes.
Analogously, we define $\fno_{g,h}(m,n)$ as the maximum entanglement number of any two systems
$A$ and $B$ of $m$ and $n$ curves, respectively, on a nonorientable surface of genus $g$ with $h$ holes.
It is easy to see that $f$ and $\fno$ are nondecreasing in  $m$ and $n$, which we will
often use in the sequel.

To give the reader some intuition about the problem,
let us illustrate which re-drawings are possible with
a $\partial$-automorphism and which are not.
In the example of Fig.~\ref{f:intro-ex}, it is clear that
the two crossings of $\beta_3$ with $\alpha_3$ can be avoided
by sliding $\beta_3$ aside.\footnote{This corresponds to an \emph{isotopy} of the surface
that fixes the boundary pointwise.} It is perhaps less obvious that the crossings of $\beta_2$ can also be eliminated: 
To picture a suitable $\partial$-automorphism, one can think of an annular
region in the interior of $\surf$, shaded darkly in Fig.~\ref{f:intro-ex}~(a), that surrounds the left hole 
and $\beta_1$ and contains most of the spiral formed by $\beta_2$. 
Then we cut $\surf$ along the outer boundary of that annular region, 
twist the region two times (so that the spiral is unwound), 
and then we glue the outer boundary back. 
Here is an 
example of a single twist of an annulus;
straight-line curves on the left are transformed to spirals on the right
\ifGD. \else (this kind of homeomorphism
is often called a \emph{Dehn twist}).\footnote{Formally, if we consider the circle $S^1=\R/2\pi \Z$ parameterized by angle,
then a single Dehn twist of the standard annulus $\mathcal{A}=S^1\times [0,1]$ is the $\partial$-automorphism of $\mathcal{A}$
given by $(\theta,r)\mapsto(\theta+2\pi r,r)$. Being a $\partial$-automorphism of the annulus, a Dehn twist of an annular region 
contained in the interior of a surface $\surf$ can be extended to a $\partial$-automorphism of $\surf$ by defining it to be the identity map
outside the annular region.} \fi
\immfigw{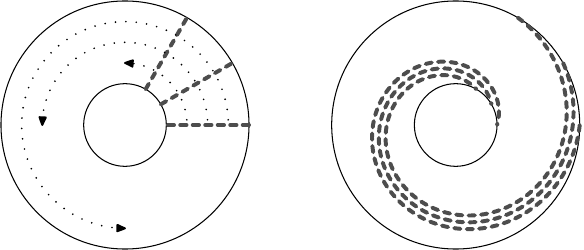}{7cm}


On the other hand, it is impossible to eliminate the crossings
of $\beta_1$ or $\beta_3$ with $\alpha_2$ by a $\partial$-automorphism.
For example, we cannot re-route $\beta_1$ to go around the right hole
and thus avoid $\alpha_2$, since this re-drawing is not induced by
any $\partial$-automorphism $\varphi$: indeed, $\beta_1$ separates the point $x$
on the boundary of left hole from
 the right hole, whereas $\alpha_2$ does not separate them; 
therefore, 
the curve $\alpha_2$ has to intersect $\varphi(\beta_1)$ at least
twice, once when it leaves the component containing $x$ and once when it
returns to this component.


A rather special case of our problem, with $m=n=1$ and only closed curves,
was already considered by Lickorish \cite{lickorish}, who 
showed that the intersection of a pair of simple
closed curves can be simplified via Dehn twists (and thus a
 $\partial$-automorphism)
so that they meet at most twice (also see Stillwell \cite{stillwell}).
The case with $m=1$, $n$ arbitrary, only closed curves, and $\surf$
possibly nonorientable was proposed in 2010 as a Mathoverflow
question \cite{overflow} by T.~Huynh. In an answer A.~Putman
proposes an approach via the ``change of coordinates principle''
(see, e.g., \cite[Sec.~1.3]{farbMargalit}), which relies on the
classification of 2-dimensional surfaces---we will also use it at some
points in our argument.

\heading{The results. } 
A natural idea for bounding $f_{g,h}(m,n)$ and $\fno_{g,h}(m,n)$
is to proceed by induction, employing the change of coordinates
principle mentioned above. This does indeed lead to finite bounds,
but the various induction schemes we have tried always led to 
bounds at least exponential in one of $m,n$. Independently of 
our work, Geelen, Huynh, and Richter 
\cite{Geelen:Explicit-bounds-for-graph-minors-2013} also recently proved bounds of this kind; see
the discussion below. Partially influenced by the results on exponentially many
intersections in representations of string graphs and similar
objects (see \cite{km-igs-94,string-np}), we first suspected that
an exponential behavior might be unavoidable. Then, however,
we found, using a very different approach, that polynomial
bounds actually do hold.


For planar $\surf$,
i.e., $g=0$, we obtain an asymptotically tight bound:

\begin{theorem}
\label{t:bound_planar}
For planar $\surf$, we have
$
f_{0,h}(m,n) = O(mn),
$
independent of $h$.
\end{theorem}

Here and in the sequel, the constants implicit in the $O$-notation are
absolute, independent of $g$ and~$h$.

A simple example providing a lower bound of $2mn$ is obtained, e.g.,
by replicating $\alpha_2$ in Fig.~\ref{f:intro-ex} $m$ times
and $\beta_1$ $n$ times. \ifGD\else We currently have no example forcing
more than $2mn$ intersections.
\medskip \fi

In general, we obtain the following bounds:
\begin{theorem}
\label{t:bound_main}
\begin{enumerate} 
\item[\textup{(i)}] For the orientable case,
\ifGD$\else\[\fi
f_{g,h}(m,n) = O((m+n)^4).
\ifGD$\else\]\fi
\item[\textup{(ii)}]
For the nonorientable case,
\ifGD$\else\[\fi
\fno_{g,h}(m,n) = O((m+n)^4).
\ifGD$\else\]\fi
\end{enumerate}
\end{theorem}

%
%
%

Both parts of Theorems~\ref{t:bound_main} are derived  from the planar case, Theorem~\ref{t:bound_planar}.
In the orientable case, we use the following results
on  genus reduction. For a convenient notation, let us
set $\maxmn=\max(m,n)$.

\begin{proposition}[Orientable genus reductions]
\label{p:bound_g}
\vspace{-2mm}
\begin{enumerate}
\item[\rm (i)] For $g>\maxmn$, we have 
\ifGD$\else\[\fi
f_{g,h}(m,n) \leq f_{\maxmn, g+h -\maxmn}(m,n).
\ifGD$\else\]\fi
\item[\rm (ii)]
$f_{g,h}(m,n) \leq f_{0,h + 1}(cg(m+g),cg(n+g))$ for a suitable
 constant $c > 0$.
\end{enumerate}
\end{proposition}

\ifGD Theorem~\ref{t:bound_main}~(i), the orientable case,
follows immediately from Proposition~\ref{p:bound_g}
and the planar bound.
\else
To derive Theorem~\ref{t:bound_main}~(i),
for $g>\maxmn$, we use Proposition~\ref{p:bound_g}(i), then~(ii),
and then the planar bound:
$  f_{g,h}(m,n) \leq f_{\maxmn, g + h - \maxmn}(m,n) 
 \leq f_{0, g + h + 1 - \maxmn}(2c\maxmn^2, 2c\maxmn^2)=O(\maxmn^4)$.
For $g\le \maxmn$, we omit the first step.
\medskip
\fi

In the nonorientable case, Theorem~\ref{t:bound_main}~(ii)
 is derived in two steps. First, analogous to
 Proposition~\ref{p:bound_g}~(i), we have the following reduction:
\begin{proposition}[Nonorientable genus reduction]
  \label{p:bound_g_nonorientable}
 For $g > 4\maxmn + 2$, we have
\ifGD$\else\[\fi
\fno_{g,h}(m,n) \leq \fno_{g',h'}(m,n),
\ifGD$\else\]\fi
where $g'=4\maxmn + 2 - (g \smod 2)$ and $h'=h + \lceil g/2
\rceil - 2\maxmn - 1$. 
\end{proposition}

The second step is a reduction to the orientable case.
\begin{proposition}[Orientability reduction]
  \label{p:or_red}
There is a constant $c$ such that
\ifGD$\else\[\fi
\hat f_{g,h}(m,n) \leq f_{\lfloor (g-1)/2 \rfloor,h+1 + (g \smod 2)} (c(g+m),
c(g+n)).
\ifGD$\else\]\fi
\end{proposition}
\ifGD\else
Now we can derive Theorem~\ref{t:bound_main}~(ii). 
We set $\maxmn := \max(m,n)$.
For $g> 4\maxmn + 2$, we use Proposition~\ref{p:bound_g_nonorientable}, then
Proposition~\ref{p:or_red}. We also use
monotonicity of the entanglement numbers in $m$ and $n$.
We obtain
$ \hat f_{g,h}(m,n) \leq  \hat f_{4\maxmn + 2 - (g \smod 2), \vartheta_1(g,h,m,n)}(m,n) \leq 
f_{2\maxmn,\vartheta_2(g,h,m,n)}(6c\maxmn, 6c\maxmn)$
where $\vartheta_1$ and $\vartheta_2$ are functions that, for simplicity, we do
not evaluate explicitly.
Then we use Proposition~\ref{p:bound_g} and the planar bound,
Theorem~\ref{t:bound_planar}, to obtain an $O(\maxmn^4)$ bound similarly as in the
orientable case.
For $g\le 4\maxmn + 2$, we omit the first step.
\fi
Table~\ref{f:suma} summarizes the proof of Theorem~\ref{t:bound_main}.
\medskip

\begin{table}
\framedpar{\begin{enumerate}
\item For a planar surface, temporarily remove the holes not incident
to any $\alpha_i$ or $\beta_j$, and contract the remaining ``active''
holes, augment the resulting planar graphs to make them 3-connected.
Make a simultaneous plane
drawing of the resulting planar graphs $G_1$ and $G_2$
with every edge of $G_1$ intersecting every edge of $G_2$
at most $O(1)$ times. Decontract the active holes and put the
remaining holes back into appropriate faces
 (Theorem~\ref{t:bound_planar}; 
Section~\ref{s:planar}).
\item If the genus is larger than $c(m+n)$, find handles or cross-caps
avoided by the $\alpha_i$ and $\beta_j$, 
temporarily remove them, untangle the $\alpha_i$ and $\beta_j$,
and put the handles or cross-caps back (Propositions~\ref{p:bound_g}~(i)
and~\ref{p:bound_g_nonorientable}; 
\ifGD the proofs are omitted from this extended abstract
\else 
Section~\ref{s:gen_O(m+n)}\fi).
\item If the surface is nonorientable, make it orientable by cutting
along a suitable curve that intersects the $\alpha_i$ and $\beta_j$
 at most $O(m+n)$ times,
untangle the resulting pieces
of the $\alpha_i$ and $\beta_j$, and glue back (Proposition~\ref{p:or_red}%
\ifGD
\else
; Section~\ref{s:non_or}\fi).
\item Make the surface planar by cutting along a suitable system
of curves (canonical system of loops), untangle
the resulting pieces 
of the $\alpha_i$ and $\beta_j$, and glue back 
(Proposition~\ref{p:bound_g}~(ii)%
 \ifGD
\else; Section~\ref{s:redu2}\fi).
\end{enumerate}}
\caption{A summary of the proof.}
\label{f:suma}
\end{table}

\heading{Motivation.} We were led to the question concerning untangling curves on
surfaces while working on a project on $3$-manifolds and embeddings. Specifically, 
we are interested in an algorithm for the following problem: given a $3$-manifold $M$ 
with boundary, does $M$ embed in the $3$-sphere? A special case of this problem, 
with the boundary of $M$ a torus, was solved in \cite{JacoSedgw}.
The general version of the problem is motivated, in turn, by the question of algorithmically
testing the embeddability of a $2$-dimensional simplicial complex
in $\R^3$; see \cite{MatousekTancerWagner:HardnessEmbeddings-2011}.

Very recently, we showed that these embeddability problems are algorithmically decidable, see
\cite{emb3}. For the proof, we use the following upper bound on $f_{g,h}(m,n)$ and $\hat f_{g,h}(m,n)$,
which we state here as a separate corollary in the specific form used in  \cite{emb3}, for convenience
of reference.

\begin{corollary} Both  $f_{g,h}(m,n)$ and $\fno_{g,h}(m,n)$ are bounded from above by 
$K(g)mn$, where $K(g)$ is a computable function of $g$, independent of $h$ (in fact, $K(g)=O(g^4)$).
\end{corollary}

\begin{proof}
By Thm.~\ref{t:bound_planar}, for planar $\surf$, we have $f_{0,h}=O(mn)$. 
By Prop.~\ref{p:bound_g}~(ii), in case of an orientable surface of arbitrary genus, 
$f_{g,h}(m,n)\le f_{0,h+1}(cg(m+g),cg(n+g))=
O(g^2(m+g)(n+g))=O(g^4mn)$. For the nonorientable
case, Prop.~\ref{p:or_red} gives $\fno_{g,h}(m,n)\le f_{\lfloor (g-1)/2\rfloor,h+1+(g \mod 2)}(c(m+g),c(n+g))=
O(g^4 mn)$ as well.
\end{proof}

Independently of the application to embeddability, we consider
the problem investigated in this paper interesting in itself
and contributing to a better understanding of combinatorial properties
of curves on surfaces.

As mentioned above, the question studied in the present paper has also been 
investigated independently by Geelen, Huynh, and Richter \cite{Geelen:Explicit-bounds-for-graph-minors-2013}, with a rather different and
very strong motivation stemming from the theory of graph minors,
namely the question of obtaining explicit upper bounds for the graph minor algorithms of Robertson and Seymour. 
Phrased in the language of the present paper, 
Geelen et al.~\cite[Theorem~2.1]{Geelen:Explicit-bounds-for-graph-minors-2013} show that $f_{g,h}(m,n)$ and $\fno_{g,h}(m,n)$ are both bounded by $n 3^m$,
but only under the assumption that 
$\surf\setminus(\beta_1\cup\cdots\cup\beta_n)$ is connected.\footnote{We remark that without this additional assumption, 
the bounds proved by Geelen et al. (or even weaker ones of the form $K(g,m)n$) could also be used for the application to the algorithmic embeddability problem, but due to the extra assumption their results cannot be directly applied to \cite{emb3} (even though it might
be possible to remove the extra assumption).} 

\ifGD\else 
\heading{Further work. } 
We suspect that the bound in Theorem~\ref{t:bound_main} should also be $O(mn)$.
The possible weak point of the current proof is the reduction
in Proposition~\ref{p:bound_g}(ii), from genus comparable to $m+n$ to
the planar case. 

This reduction uses a result of the following kind:
given a graph $G$ with $n$ edges embedded on a compact $2$-manifold
$\surf$ of genus $g$ (without boundary), one can construct a system of curves
on $\surf$ such that cutting $\surf$ along these curves yields one or
several planar surfaces, and at the same time, the curves have a bounded
number of crossings with the edges of $G$ (see Section~\ref{s:redu2}).
Concretely, we use a result of Lazarus et al.~\cite{Lazarus-al}, where
the system of curves  is of a special kind, forming a \emph{canonical
system of loops}. (This result is in fact essentially due to Vegter and
Yap~\cite{VegterYap}; however, the formulation in~\cite{Lazarus-al} is more
convenient for our purposes.) Their result is asymptotically optimal
for a canonical system of loops, but it may be possible to improve
it for other systems of curves. This and similar questions have
been studied in the literature, mostly in algorithmic context
 (see, e.g., \cite{CabelloMohar,Demaine-al,CdVthesis,CdVHabilitation} 
for some of the
relevant works), but we haven't found any existing result superior to 
that of Lazarus et al.\ for our purposes.
\fi

\section{Planar Surfaces}
\label{s:planar}

In this section  we prove Theorem~\ref{t:bound_planar}.
In the proof we use the following basic fact (see, e.g.,~\cite{MoharThomassen}).

\begin{lemma}\label{l:drawings-same}
If $G$ is a maximal planar simple graph
(a triangulation), then for every two planar drawings of $G$ in $S^2$
there is an automorphism $\psi$ of $S^2$ converting one of the drawings into
the other (and preserving the labeling of the vertices and edges).
Moreover, if an edge $e$ is drawn by the same arc in both of the
drawings, w.l.o.g.\ we may assume that $\psi$ fixes this arc pointwise.
\end{lemma}

Let us introduce the following piece of terminology.
Let $G$ be as in the lemma, and let $D_G$, $D'_G$ be two planar drawings
of $G$. We say that $D_G,D'_G$ are \emph{directly equivalent}
if there is an orientation-preserving automorphism of $S^2$
mapping $D_G$ to $D'_G$, and we call $D_G,D'_G$  \emph{mirror-equivalent}
if there is an orientation-reversing automorphism of $S^2$ converting
$D_G$ into $D'_G$.

We will also rely on a result concerning simultaneous planar embeddings;
see \cite{sim-surv}.
Let $V$ be a vertex set and let $G_1=(V,E_1)$ and $G_2=(V,E_2)$ be two planar
graphs on~$V$. A planar drawing $D_{G_1}$ of $G_1$ and a planar drawing
$D_{G_2}$ of $G_2$ are said to form a \emph{simultaneous embedding}
of $G_1$ and $G_2$ if each vertex $v\in V$ is represented by the
same point in the plane in both $D_{G_1}$ and $D_{G_2}$ (in particular any edge
drawn in $D_{G_1}$ may intersect any edge drawn in $D_{G_2}$).

We note that $G_1$ and $G_2$ may have common edges, but they are
not required to be drawn in the same way in $D_{G_1}$ and in $D_{G_2}$.
If this requirement is added, one speaks of a \emph{simultaneous
embedding with fixed edges}. There are pairs of planar graphs
known that do not admit any simultaneous embedding with
fixed edges (and consequently, no simultaneous straight-line
embedding). An important step in our approach is very similar to the proof of 
the following result.

\begin{theorem}[Erten and Kobourov \cite{EK}]\label{t:EK}
Every two planar graphs $G_1=(V,E_1)$ and $G_2=(V,E_2)$
admit a simultaneous embedding in which every edge
is drawn as a polygonal line with at most $3$ bends.
\end{theorem}


We will need the following result, which follows easily from the proof
given in~\cite{EK}. For the reader's convenience,
instead of just pointing out the necessary modifications,
we present a full proof.

\begin{theorem}
\label{t:spiked}
Every two planar graphs $G_1=(V,E_1)$ and $G_2=(V,E_2)$
admit a simultaneous, piecewise linear embedding
in which each edge of $G_1$ and each edge of $G_2$ intersect 
at least once and at most
$C$ times, for a suitable constant $C$.\footnote{An obvious bound from the
proof is $C\le 36$, since 
every edge in this embedding is drawn using at most $5$
bends. By a more careful inspection, one can easily
get $C\le 25$, and a further improvement is probably possible.} 

In addition, if both $G_1$ and $G_2$ are
maximal planar graphs, let us fix a planar drawing $D'_{G_1}$
of $G_1$ and a planar drawing $D'_{G_2}$ of $G_2$.
 The planar drawing of $G_1$ in the
simultaneous embedding can be required to be either
directly equivalent to $D'_{G_1}$, or mirror-equivalent to it,
and similarly for the drawing of $G_2$ (each of the four combinations
can be prescribed). 
\end{theorem}

\begin{proof}
 For the beginning, we assume that both graphs are Hamiltonian. Later on, we will drop this assumption. 

Let $v_1, v_2, \dots, v_n$ be the order of the vertices as they appear on
(some) Hamiltonian cycle $H_1$ of $G_1$. Since the vertex set $V$ is common for $G_1$
and $G_2$, there is a permutation $\pi \in S(n)$ such that $v_{\pi(1)}, \dots,
v_{\pi(n)}$ is the order of the vertices as they appear on some Hamiltonian
cycle $H_2$ of $G_2$.

We draw the vertex $v_i$ in the grid point $p_i = (i, \pi(i))$, $i=1,2,\ldots,n$.
 Let $S$ be the square $[1,n]\times [1, n]$. A \emph{bispiked} curve is
an 
$x$-monotone polygonal curve with two bends such that it starts inside $S$; the
first bend is above $S$, the second bend is below $S$ and it finishes in $S$
again.

The $n-1$ edges $v_iv_{i + 1}$, of $H_1$, $i=1,2,\ldots,n-1$, are drawn as
bispiked curves starting in $p_i$ and finishing in $p_{i+1}$. In order to
distinguish edges and their drawings, we denote these bispiked curves by
$c(i,i+1)$. 

Similarly, we draw the
edges $v_{\pi(i)}v_{\pi(i+1)}$ of $H_2$, $i=1,2,\ldots,n-1$,
as $y$-monotone analogs 
of bispiked curves, where the first bend is on the left of $S$ 
and the second is on the right of $S$; here is an example:
\immfigw{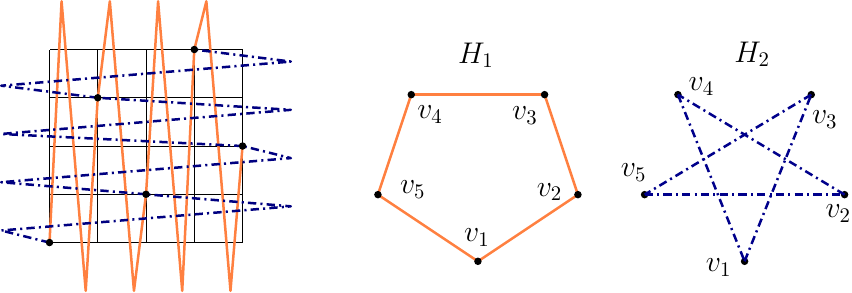}{9cm}


We continue only with description of how to draw $G_1$; $G_2$ is
drawn analogously with the grid rotated by $90$ degrees. 


Let $D'_{G_1}$ be a planar drawing of $G_1$. Every edge from $E_1$ that 
is not contained in $H_1$ is drawn either inside $D'_{H_1}$ or outside. 
Thus, we
split $E_1 \setminus E(H_1)$ into two sets $E'_1$ and $E''_1$.

Let $P_0$ be the polygonal path obtained by concatenation of the
curves $c(1, 2)$,
$c(2,3),\dots$, $c(n-1, n)$. 
Now our task is to draw the edges of $E'_1 \cup \{v_1v_n\}$ 
as bispiked curves, all
above $P_0$, and then the edges of $E''_1$ below~$P_0$.

We start with $E'_1$ and we draw edges from it one by one,
in a suitably chosen
order, while keeping the following properties. 

\begin{enumerate}
\item[(P1)] Every edge $v_iv_j$, where $i < j$, is drawn as a bispiked curve $c(i,j)$ 
starting in $p_i$ and ending in $p_j$. 
\item[(P2)]
 The $x$-coordinate of the second bend of $c(i,j)$ belongs to the interval
$[j-1,j]$.
\item[(P3)]
 The polygonal curve $P_k$ that we see from above after drawing the $k$th edge is
obtained as a concatenation of some curves $c(1, i_1), c(i_1, i_2), \dots,
c(i_\ell, n)$.
\end{enumerate}
Here is an illustration; the square $S$ is deformed for the
purposes of the drawing:
\immfigw{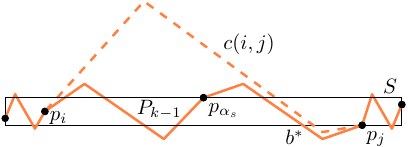}{5cm}


Initially, before drawing the first edge,
 the properties are obviously satisfied. 

Let us assume that we have already drawn $k-1$ edges of $E'_1$,
 and let us focus on drawing the $k$th edge. 
Let $e = v_iv_j \in E'_1$ be an edge that is not
yet drawn and such that all edges below $e$ are already drawn,
where ``below $e$'' means 
all edges $v_{i'}v_{j'} \in E'_1$ with $i \leq i' < j' \leq j$, 
$(i,j) \neq (i',j')$.
(This choice ensures that we will draw all edges of $E'_1$.) 

Since $D'_{G_1}$ is a planar drawing, we know that there is no edge
$v_{i'}v_{j'} \in E'_1$ with $i < i' < j < j'$ or $i' < i < j' < j$, and so
the points $p_i$ and $p_j$ have to belong to $P_{k-1}$. The subpath $P'$ of
$P_{k-1}$ between $p_i$ and $p_j$ is the concatenation of curves $c(i,
\alpha_1), c(\alpha_1, \alpha_2), \dots, c(\alpha_s, j)$ as
in the inductive assumptions.
In particular, the $x$-coordinate of the second bend $b^*$ of $c(\alpha_s,j)$
belongs to the interval $[j-1,j]$. We draw $c(i,j)$ as follows:
The second bend of $c(i,j)$ is slightly above $b^*$ but still below the square
$S$. The first bend of $S$ is sufficiently high above $S$ 
(with the $x$-coordinate somewhere between $i$ and $j-1$)
 so that the resulting bispiked curve $c(i,j)$ does not intersect $P_{k-1}$. 
The properties (P1) and (P2) are
obviously satisfied by the construction. For (P3), the path $P_k$
is obtained from $P_{k-1}$ by replacing $P'$ with $c(i,j)$.

After drawing the edges of $E'_1$, we 
draw $v_1v_n$ in the same way. Then we draw
the edges of $E''_1$ in a similar manner as those of $E'_1$,
this time as bispiked curves below $P_0$. This finishes
the construction for Hamiltonian graphs. 

Now we describe how to adjust this
construction for non-Hamiltonian graphs, in the spirit of~\cite{EK}.

First we add edges to $G_1$ and $G_2$ so that they become
 planar triangulations. 
This step does not affect the construction at all, except that
we remove these edges in the final drawing.

Next, we subdivide some of the edges of $G_i$ with \emph{dummy}
vertices. Moreover, we attach two new \emph{extra} edges to each dummy vertex,
as in the following illustration:
\immfigw{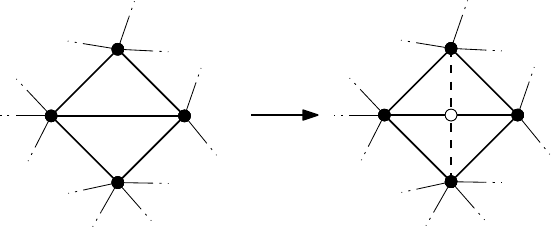}{4cm}
By choosing the subdivided edges suitably, 
one can obtain a $4$-connected, and thus Hamiltonian, graph; see
\cite[Proof of Theorem 2]{EK} for details (this idea previously comes
from~\cite{KaufmannWiese:RemovingHamiltonicity}). An
important property of this construction is that each edge of $G_i$ is
subdivided at most once. 

In this way, we obtain new Hamiltonian graphs $G'_1$ and $G'_2$, 
for which we want to construct a simultaneous drawing as  
in the first part of the proof.
A little catch is that $G'_1$ and $G'_2$ do not have same vertex sets,
but this is easy to fix.
 Let $d_i$ be the number of dummy vertices of $G'_i$, $i=1,2$,
and say that $d_1\ge d_2$.
We pair the $d_2$ dummy vertices of $G'_2$ with some of the
dummy vertices of $G'_1$. Then we
iteratively add $d_1 - d_2$ new triangles to $G'_2$,
 attaching each of them to an edge of
a Hamiltonian cycle. This operation keeps Hamiltonicity and
introduces $d_1 - d_2$ new vertices, which can be matched with
the remaining $d_1-d_2$ dummy vertices in~$G'_1$.

After drawing resulting graphs, we remove all extra dummy vertices and extra
edges added while introducing dummy vertices. An original edge $e$ that was
subdivided by a dummy vertex is now drawn as a concatenation 
of two bispiked curves. Therefore,
each edge is drawn with at most $5$ bends.

Two edges with 5 bends each may in general have at most 36 intersections,
but in our case, there can be at most 25 intersections, since the union
of the two segments before and after a dummy vertex is both $x$-monotone
and $y$-monotone.

Because of the bispiked drawing of all edges, it is also clear that every edge
of $G_1$ crosses every edge of $G_2$ at least once.


Finally, the requirements on directly equivalent or mirror-equivalent 
drawings can easily be fulfilled by interchanging the role
of top and bottom in the drawing of $G_1$ or left and right
in the drawing of $G_2$. Theorem~\ref{t:spiked} is proved.
\end{proof}

\begin{proof}[Proof of Theorem~\ref{t:bound_planar}.]
Let a planar surface $\surf$ and the curves $\alpha_1,\ldots,\alpha_m$,
$\beta_1,\ldots,\beta_n$ be given; we assume that $\surf$
is a subset of $S^2$. Furthermore; by eventually applying some
$\partial$-automorphism moving the curves $\beta_j$, we can assume that for
every $i$ and $j$ the curves $\alpha_i$ and $\beta_j$ meet on the boundary in
the endpoints or in the interior transversally and in
a finite number of points. From this we construct
a set $V$ of $O(m+n)$ vertices in $S^2$ and planar drawings $D_{G_1}$
and $D_{G_2}$ of two simple graphs $G_1=(V,E_1)$ 
and ${G_2}=(V,E_2)$ in $S^2$, as follows.
\begin{enumerate}
\item We put all endpoints of the $\alpha_i$ and of the $\beta_j$
into~$V$ (note that some of them can be shared).
\item We choose a new vertex in the interior of each  $\alpha_i$
and each $\beta_j$, or two distinct vertices if $\alpha_i$ or $\beta_j$
is a loop with a single endpoint, or three vertices if $\alpha_i$ or $\beta_j$
is a closed curve, and we add all of these vertices to~$V$.
These new vertices are all distinct and do not lie on any
curves other than where they were placed.
\item If the boundary of a hole in $\surf$ already 
contains a vertex introduced so far, we add more vertices 
so that it contains at least $3$ vertices of~$V$.
This finishes the construction of~$V$.
\item To define the edge set $E_1=E(G_1)$ and the planar drawing $D_{G_1}$,
we take the portions of the curves $\alpha_1,\ldots,\alpha_m$
between consecutive vertices of $V$ as edges of~$E_1$.
Similarly, we make the arcs of the boundaries of the holes
into edges in $E_1$; these will be called the \emph{hole edges}.
By the choice of the vertex set $V$ above, this yields
a simple plane graph.
\item
Then we add new edges to $E_1$
so that we obtain a drawing $D_{G_1}$ in $S^2$ of a maximal 
planar simple graph $G_1$ (i.e., a triangulation)
on the vertex set~$V$. While choosing these edges, we make sure
that all holes containing no vertices of $G$ lie in faces of
$D_{G_1}$ adjacent to some of the $\alpha_i$.
New edges drawn in the interior
of a hole are also called \emph{hole edges}.
\item We construct ${G_2}=(V,E_2)$ and $D_{G_2}$ analogously, using the curves
$\beta_1,\ldots,\beta_n$.  We make sure that all hole edges are common to
$G_1$ and~${G_2}$.
\end{enumerate}

After this construction, each hole of $\surf$ contains either no vertex of $V$ on
its boundary or at least three vertices. In the former case, 
we speak of an
\emph{inner} hole, and in the latter case, of a \emph{subdivided hole}. 
A face $f$ of $D_{G_1}$ or $D_{G_2}$ is a \emph{non-hole face} if it is not contained
in a subdivided hole.
An inner hole $H$ has its \emph{signature}, which is a pair
$(f_1, f_2)$, where $f_1$ is the unique non-hole face of $D_{G_1}$ containing
$H$, and
$f_{2}$ is the unique non-hole face of $D_{G_2}$ containing
$H$.\footnote{Classifying inner holes according to the signature helps us
to obtain a bound independent on the number of holes. Inner holes with same
signature are all treated in the same way, independent of their number.} 
By the construction, each $f_1$ appearing in a signature is adjacent
to some $\alpha_i$, and each $f_2$ is adjacent to some~$\beta_j$.

In the following claim, we will consider different drawings $D'_{G_1}$ and $D'_{G_2}$
for $G_1$ and ${G_2}$.  By Lemma~\ref{l:drawings-same}, the
faces of $D_{G_1}$ are in one-to-one
correspondence with the faces of $D'_{G_1}$. 
For a face $f_1$ of $D_{G_1}$, we denote the
corresponding face by $f'_1$, and similarly for a face $f_2$
of $D_{G_2}$ and~$f'_2$.

\ifGD
\begin{numclaim}
\else
\begin{claim}
\fi
\label{c:same-holes}
The graphs $G_1$ and ${G_2}$ as above have planar drawings
$D'_{G_1}$ and $D'_{G_2}$, respectively, that form a simultaneous embedding
in which each edge of $G_1$ crosses each edge of $G_2$ at most $C$ times,
for a suitable constant $C$;
moreover, $D'_{G_1}$ is directly equivalent to $D_{G_1}$; $D'_{G_2}$ is directly equivalent
to $D_{G_2}$; all hole edges are
drawn in the same way in $D'_{G_1}$ and~$D'_{G_2}$; and
whenever $(f_1,f_2)$ is a signature of an inner hole,
the interior of the intersection $f'_1 \cap f'_2$ is nonempty.
\ifGD
\end{numclaim}
\else
\end{claim}
\fi

We postpone the proof of Claim~\ref{c:same-holes}, and we first
finish the proof of Theorem~\ref{t:bound_planar} assuming this claim.

For each inner hole $H$ with signature $(f_1, f_2)$, we introduce a closed
disk $B_H$ in the interior of $f'_1 \cap f'_2$. We require that
these disks are pairwise disjoint. In sequel, we consider holes as subsets of
$S^2$ homeomorphic to closed disks (in particular, a hole $H$ 
intersects $\surf$ in~$\partial H$). 


\ifGD
\begin{numclaim}
\else
\begin{claim}
\fi
\label{c:tg}
There is an orientation-preserving automorphism $\varphi_1$ of $S^2$
transforming every inner hole $H$ to $B_H$ and $D_{G_1}$ to
$D'_{G_1}$. 
\ifGD
\end{numclaim}
\else
\end{claim}
\fi

\begin{proof}
Using Lemma~\ref{l:drawings-same} again, there is an orientation-preserving
 automorphism $\psi_1$ transforming $D_{G_1}$ into $D'_{G_1}$ 
(since $D_{G_1}$ and $D'_{G_1}$ are directly equivalent).

Let $f_1$ be a face of $D_{G_1}$. The interior of $f'_1$ contains images
$\psi_1(H)$ of all holes $H$ with signature $(f_1, \cdot)$,
and it also contains the disks $B_{H}$ for these holes. 
Therefore, there is a
boundary- and orientation-preserving automorphism of $f'_1$ that maps 
each $\psi_1(H)$ to~$B_{H}$.

By composing these automorphisms on every $f'_1$ separately, we have an
orientation-preserving automorphism $\psi_2$ fixing $D'_{G_1}$ and transforming
each $\psi_1(H)$ to $B_{H}$. 
The required automorphism is $\varphi_1 = \psi_2 \psi_1$.
\end{proof}

\ifGD
\begin{numclaim}
\else
\begin{claim}
\fi
\label{c:th}
There is an orientation-preserving automorphism $\varphi_2$ of $S^2$
that fixes hole edges (of subdivided holes), fixes $B_H$
for every inner hole $H$, and transforms $\varphi_1(D_{G_2})$ to $D'_{G_2}$.
\ifGD
\end{numclaim}
\else
\end{claim}
\fi

\begin{proof}
By Lemma~\ref{l:drawings-same} there is an orientation-preserving automorphism
$\psi_3$ of $S^2$ that fixes hole edges and transforms $\varphi_1(D_{G_2})$ to
$D'_{G_2}$. 

If an inner hole $H$ has a signature $(\cdot, f_2)$, then
both $\psi_3(B_H)$ and $B_{H}$ belong to the interior of $f'_2$.
Therefore,  as in the proof of the previous claim, there is an
orientation-preserving homeomorphism $\psi_4$ that 
fixes $D'_{G_2}$ and transforms
$\psi_3(B_H)$ to $B_{H}$. We can even require that $\psi_4 
\psi_3$ is identical on $B_{H}$. We set 
$\varphi_2 := \psi_4  \psi_3$.
\end{proof}

To finish the proof of Theorem~\ref{t:bound_planar},
we set $\varphi = \varphi_1^{-1}  \varphi_2 
\varphi_1$. We need that $\varphi$ fixes the holes (inner or subdivided) and
that $\alpha_1, \dots, \alpha_m$ and $\varphi(\beta_1), \dots,
\varphi_1(\beta_m)$ have $O(mn)$ intersections. It is routine to check all the
properties:

If $H$ is a hole (inner or subdivided), then $\varphi_2$ fixes 
$\partial \varphi_1 (H)$. Therefore, $\varphi$ also restricts to
a $\partial$-automorphism of~$\surf$.

The collections of curves $\alpha_1, \dots, \alpha_m$ and $\varphi(\beta_1),
 \dots, (\beta_m)$ have same intersection properties as
the collections $\varphi_1(\alpha_1),\dots$, $\varphi_1(\alpha_m)$ and
$\varphi_2(\varphi_1(\beta_1)),\dots$, $\varphi_2(\varphi_1(\beta_m))$. 
Since each $\alpha_i$ and each $\beta_j$ was subdivided at most three times
in the construction, by
Claims~\ref{c:same-holes}, \ref{c:tg}, and~\ref{c:th}, these collections have
at most $O(mn)$
intersections. The proof of the theorem is finished, except for
Claim~\ref{c:same-holes}.
\end{proof}

\begin{proof}[Proof of Claim~\ref{c:same-holes}.] Given $G_1$ 
and ${G_2}$, we form
auxiliary planar graphs $\tilde G_1$ and $\tilde G_2$ on a vertex set~$\tilde V$
by contracting all hole edges and removing the resulting loops and multiple
edges. We note that a loop cannot arise from an edge that
was a part of some $\alpha_i$ or~$\beta_j$.

Then we consider planar drawings $D_{\tilde G_1}$ and
$D_{\tilde G_2}$ forming a simultaneous embedding 
as in Theorem~\ref{t:spiked}, with each edge of $\tilde G_1$ crossing each
edge of $\tilde G_2$ at least once and most a constant number of times.

Let $v_H\in \tilde V$ be the vertex obtained by contracting the
hole edges on the boundary of a hole~$H$. Since the drawings
$D_{\tilde G_1}$ and $D_{\tilde G_2}$ are piecewise linear,
in a sufficiently small neighborhood of $v_{H}$ the edges
are drawn as radial segments.

We would like to replace $v_H$ by a small circle and thus turn
the drawings $D_{\tilde G_1}$, $D_{\tilde G_2}$ into the required
drawings $D'_{G_1}$, $D'_{G_2}$. But a potential problem is that
the edges in $D_{\tilde G_1}$, $D_{\tilde G_2}$ may enter
$v_H$ in a wrong cyclic order.

We claim that the edges in $D_{\tilde G_1}$ entering $v_H$ have the
same cyclic ordering around $v_H$ as the corresponding
edges around the hole $H$ in the drawing $D_{G_1}$. Indeed,
by contracting the hole edges in the drawing $D_{G_1}$, we obtain
a planar drawing $D^*_{\tilde G_1}$ of $\tilde G_1$ in which the cyclic order
around $v_H$ is the same as the cyclic order around
$H$ in $D_{G_1}$
Since $\tilde G_1$ was obtained by edge contractions from
a maximal planar graph, it is maximal as well
(since an edge contraction cannot create a non-triangular face),
and its drawing is unique up to an automorphism of $S^2$ (Lemma~\ref{l:drawings-same}).
Hence the cyclic ordering of edges around $v_H$
in $D_{\tilde G_1}$ and in $D^*_{\tilde G_1}$ is either the same
(if $D_{\tilde G_1}$ and $D^*_{\tilde G_1}$ are directly
equivalent),
or reverse (if $D_{\tilde G_1}$ and $D^*_{\tilde G_1}$ are mirror-equivalent).
However, Theorem~\ref{t:spiked} allows us to 
choose the drawing $D_{\tilde G_1}$ so that it is
directly equivalent to $D^*_{\tilde G_1}$, and then the cyclic
orderings coincide. A similar consideration applies
for the other graph~$G_2$.

The edges  of $D_{\tilde G_1}$ may still be placed to wrong positions among
the edges in $D_{\tilde G_2}$, but this can be rectified
at the price of at most one extra crossing for every pair of
edges entering $v_H$, as
the following picture indicates (the numbering specifies
the cyclic order of the edges around $H$ in $D_{G_1}\cup D_{G_2}$):
\immfig{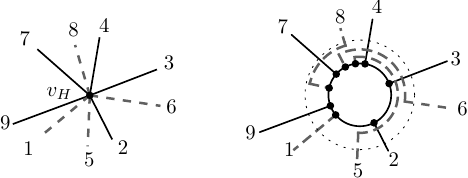}

It remains to draw the edges of $G_1$ and $G_2$
that became loops or multiple edges after the contraction of the hole
edges. Loops can be drawn along the circumference of the hole,
and multiple edges are drawn very close to the corresponding single edge.

In this way, every edge of $G_1$ still has at most a constant number
of intersections with every edge of $G_2$, and every two such edges
intersect at least once unless at least one of them became a loop
after the contraction. Consequently, whenever $(f_1,f_2)$ is a signature
of an inner hole, the corresponding faces $f_1'$ and $f_2'$ intersect.
This finishes the proof.
\end{proof}

\ifGD
\heading{Acknowledgement. }
We would like to thank the authors of \cite{Geelen:Explicit-bounds-for-graph-minors-2013} for making a draft of their paper available to us,
and, in particular, T.~Huynh for an e-mail correspondence.
\bibliographystyle{alpha}
\bibliography{cx.bib}

\clearpage
\appendix
\section*{Appendix}
\else
\fi
\section{Reducing the Genus to \boldmath $O(m+n)$}
\label{s:gen_O(m+n)}

In this section we prove Proposition~\ref{p:bound_g}(i) as well as
Proposition~\ref{p:bound_g_nonorientable}.
We begin with several definitions.

\subsection{Cutting Along Curves}
\label{subs:cutting}
Let $\surf$ be an (orientable or nonorientable) surface with boundary.
By $h(\surf)$ we denote the number of holes in $\surf$ and by $g(\surf)$ we denote
the (orientable or nonorientable) genus of $\surf$. 

Now let $\delta$ be a properly embedded curve in $\surf$ (i.e., either a simple closed curve 
that avoids the boundary $\partial\surf$, or a simple arc whose endpoints lie on $\partial\surf$).
The curve $\delta$ is called \emph{separating} if $\surf \setminus \delta$
has two components. Otherwise, $\delta$ is \emph{non-separating}.

We denote by $\cut{\surf}{\delta}$ the (possibly disconnected) 
surface obtained by cutting $\surf$ 
along $\delta$. 
If $\delta$ is non-separating, then $\cut{\surf}{\delta}$
is connected. Otherwise, $\cut{\surf}{\delta}$ has two components, which we denote
by $\cut{\surf}{\delta}^1$ and $\cut{\surf}{\delta}^2$.

Now we recall basic properties of the \emph{Euler
characteristic} of a surface. Given a triangulated surface $\surf$, the Euler
characteristic $\chi(\surf)$ is defined as the number of vertices plus number of
triangles minus the number of edges in the triangulation. It is well known that
the Euler characteristic is a topological invariant and equals $2 - 2g(\surf) -
h(\surf)$ if $\surf$ is orientable, and $2 - g(\surf) - h(\surf)$
if $\surf$ is nonorientable.

To work simultaneously with orientable and nonorientable surfaces, it is also
convenient to define the \emph{Euler genus} of $\surf$ as $g_e(\surf) := 2 -
\chi(\surf) - h(\surf)$. That is, $g_e(\surf) = g(\surf)$ if $\surf$ is
nonorientable, and $g_e(\surf) = 2g(\surf)$ if $\surf$ is orientable.

We have the following relations for the Euler characteristic:
\begin{center}
\begin{tabular}{r|c|c}
 & $\delta$ is non-separating & $\delta$ is separating \\
\hline
$\delta$ is a closed curve & $\chi(\surf) = \chi(\cut{\surf}{\delta})$ & $\chi(\surf) =
\chi(\cut{\surf}{\delta}^1) + \chi(\cut{\surf}{\delta}^2)$ \\
\hline
$\delta$ is an arc & $\chi(\surf) = \chi(\cut{\surf}{\delta}) - 1$& $\chi(\surf) =
\chi(\cut{\surf}{\delta}^1) + \chi(\cut{\surf}{\delta}^2) - 1$\\
\end{tabular}
\end{center}

The relations above also allow us to
 relate the genus of $\surf$ and the genus of
the surface(s) obtained after a cutting.

Let us call a closed curve $\delta$ in $\surf$ \emph{two-sided}
if a small closed neighborhood of $\delta$ is homeomorphic 
 to the annulus $S^1 \times [0,1]$; otherwise,
$\delta$ is \emph{one-sided} (and a small closed neighborhood
of $\delta$ is a M\"obius band).
Note that every orientable surface contains only two-sided closed curves.

\begin{lemma}
\label{l:genera}
We have the following relations for genera:
\begin{enumerate}[(a)]
  \item
If $\nsurf$ is orientable, then
    $$
g(\surf) = 
\left\{ \begin{array}{ll}
    g(\cut{\surf}{\delta}^1) + g(\cut{\surf}{\delta}^2) & \mbox{if $\delta$ is separating};\\
   g(\cut{\surf}{\delta}) & \mbox{if $\delta$ is a non-separating arc connecting}\\
               & \mbox{two different boundary components};\\
 g(\cut{\surf}{\delta}) + 1 & \mbox{if $\delta$ is a non-separating closed
  curve, or}\\
               & \mbox{a non-separating arc with both endpoints}\\
               & \mbox{in a single boundary component.}\\
        \end{array} \right.
$$
\item
  If $\nsurf$ is orientable or nonorientable, then
    $$
g_e(\surf) = 
\left\{ \begin{array}{ll}
    g_e(\cut{\surf}{\delta}^1) + g_e(\cut{\surf}{\delta}^2) & \mbox{if $\delta$ is separating};\\
   g_e(\cut{\surf}{\delta}) & \mbox{if $\delta$ is a non-separating arc connecting}\\
               & \mbox{two different boundary components};\\
 g_e(\cut{\surf}{\delta}) + 1 & \mbox{if $\delta$ is a non-separating one-sided
  closed curve}\\
g_e(\cut{\surf}{\delta}) + 2 & \mbox{if $\delta$ is a non-separating arc with both
endpoints}\\
               & \mbox{in a single boundary component, or}\\
               & \mbox{a non-separating two-sided closed curve.}

        \end{array} \right.
$$

\end{enumerate}
\end{lemma}

Note that $(b)$ implies $(a)$. However, it is still convenient to state $(a)$
separately.

\begin{proof}
A simple case analysis yields the following relations for the numbers of holes:
$$
h(\surf) = 
\left\{ \begin{array}{ll}
    h(\cut{\surf}{\delta}^1) + h(\cut{\surf}{\delta}^2) - 2 & \mbox{if $\delta$ is
a separating closed curve};\\
	 h(\cut{\surf}{\delta}) - 2 & \mbox{if $\delta$ is
a two-sided non-separating closed curve};\\
	 h(\cut{\surf}{\delta}) - 1 & \mbox{if $\delta$ is
a one-sided non-separating closed curve};\\
	 h(\cut{\surf}{\delta}^1) + h(\cut{\surf}{\delta}^2) - 1 & \mbox{if $\delta$ is
a separating arc};\\
	 h(\cut{\surf}{\delta}) + 1 & \mbox{if $\delta$ is a non-separating arc
connecting}\\
               & \mbox{two different boundary components};\\
	h(\cut{\surf}{\delta}) - 1 & \mbox{if $\delta$ is a non-separating arc with
both}\\
               & \mbox{endpoints in a single boundary component.}\\
        \end{array} \right.
$$
The proof now follows by simple computation from the table above the lemma
 and the relations $\chi(\surf) = 2 - 2g(\surf) - h(\surf)$ if $\surf$ is
 orientable and $\chi(\surf) = 2 - g_e(\surf) - h(\surf)$ if $\surf$ is
 orientable or nonorientable.
\end{proof}

\subsection{Orientable Surfaces}
Let $\surf$ be a surface, which may be orientable or nonorientable. A
\emph{handle-enclosing} curve 
is a separating closed curve $\lambda$ in $\surf$ that splits $\surf$ into
two components $\surf_{\langle \lambda \rangle}^+$ and $\surf_{\langle \lambda
\rangle}^-$ such that $\surf_{\langle \lambda \rangle}^-$ is
a \emph{torus with hole}---that is, an orientable surface of genus $1$ with one
boundary hole; here are two ways of looking at it:
\immfig{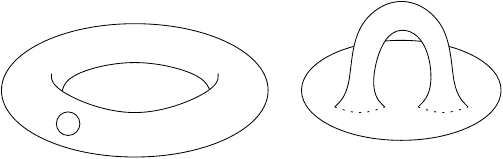}
 A system $L$ of handle-enclosing curves is \emph{independent}
if $\surf_{\langle \kappa \rangle}^- \cap \surf_{\langle \lambda \rangle}^- = \emptyset$ for every
two closed curves $\kappa,\lambda \in L$.


First we focus on proving Proposition~\ref{p:bound_g}~(i). For the remainder of this subsection,
all surfaces will be orientable.

For an orientable surface of genus $g$ with $h$ holes, we fix a
\emph{standard representation} of this surface, denoted by $\surf_{g,h}$. It
is obtained by removing interiors 
of $h$ pairwise disjoint disks $H_1, \dots, H_h$ in
the southern hemisphere of $S^2$ and by removing interiors of $g$ pairwise
disjoint disks $D_1, \dots, D_g$ in the northern hemisphere of
$S^2$ and then attaching a torus with hole along the boundary
of each $D_i$; see Fig.~\ref{f:s32}. 
Note that $\{\partial
D_i\}_{i=1}^g$ is an independent system of handle-enclosing curves.

\begin{figure}
\begin{center}
\includegraphics{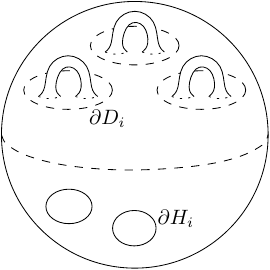}
\end{center}
\caption{The standard representation $\surf_{3,2}$.}
\label{f:s32}
\end{figure}

One of the tools we need (Lemma~\ref{l:stand_rep}) is that if we find
handle-enclosing curves in some
surface $\surf$ (of genus $g$ with $h$ holes), then we can find a homeomorphism
$\surf \to \surf_{g,h}$ mapping these curves to $\partial D_i$ extending some
given homeomorphism of the boundaries. However, we have to require a
technical condition on orientations, to be described next.

Let $\gamma_1, \dots, \gamma_h$ be a collection of the boundary curves of an
orientable surface $\surf$ (of arbitrary genus) with $h$ holes. We assume that
$\gamma_1, \dots, \gamma_h$ are also given with orientations. Since $\surf$
is orientable, it makes sense to speak of whether the orientations of $\gamma_1, \dots, \gamma_h$ 
are mutually \emph{compatible} or not: Choose and fix an orientation of $\surf$. 
Then we can say for each boundary curve $\gamma_i$ whether $\surf$ lies is on the 
right-hand side of $\gamma_i$ or on the left-hand side (with respect to the chosen orientation
of $\surf$ and the given orientation of $\gamma_i$).\footnote{If $\surf$ is smooth, for instance,
and if we choose a point $p_i$ in each $\gamma_i$, then  there are two distinguished
unit vectors in the tangent plane of $\surf$ at $p_i$: the inner normal vector $\nu_i$ of $\gamma_i$ within
$\surf$ (which is independent of any orientation), and the tangent vector $\tau_i$ of $\gamma_i$ 
(which depends on the orientation of $\gamma_i$). The orientations of the boundary curves
$\gamma_1,\ldots,\gamma_h$ are compatible if and only if each pair $(\nu_i,\tau_i)$ determines the same
orientation of $\surf$.}
\begin{lemma}
\label{l:planar_from_boundary}
  Let $\surf$ be a planar surface with $h$ holes. Let $\gamma_1, \dots,
  \gamma_h$ be the boundary curves of $\surf$ given with compatible
  orientations. Let $\zeta\colon \partial \surf \to \partial \surf_{0,h}$ be a
 homeomorphism
  such that the orientations (induced by $\zeta$) of the curves $\zeta(\gamma_1),
  \dots, \zeta(\gamma_h)$ are compatible. Then $\zeta$ can be extended to
a homeomorphism $\bar \zeta \colon \surf \to \surf_{0,h}$.
\end{lemma}

The lemma is generally known and the proof is quite straightforward. 
We keep the proof here for completeness (and for lack of a reference). Similar
remark applies to Lemma~\ref{l:stand_rep} below.

\begin{proof}
  If $h = 0$, then the claim follows immediately from the classification of
  surfaces. For $h = 1$,
 an arbitrary homeomorphism $\partial \surf \to \partial \surf_{0,h}$ 
 can be extended to a homeomorphism $\surf \to
  \surf_{0,h}$ (between disks) by `coning'.

 For $h > 1$ we prove the lemma by induction in $h$. We connect two
 (closed) boundary curves $\gamma_1$, $\gamma_2$ with an
 arc $\delta$ inside $\surf$ attached at some points $a$ and $b$ and  we also
 connect $\zeta(\gamma_1)$ and $\zeta(\gamma_2)$ inside $\surf_{0,h}$ with an
 arc $\delta'$ attached at $\zeta(a)$ and $\zeta(b)$. We cut $\surf$ and
 $\surf_{0,h}$ along $\delta$ and $\delta'$, obtaining surfaces $\surf^*$
 and $\surf^*_{0,h}$ with one hole less. 
 
 The holes $\gamma_3, \dots, \gamma_h$ are
 kept in $\surf^*$, while the holes $\gamma_1$ and $\gamma_2$ and the arc $\delta$
 in $\surf$ induce a boundary curve $\gamma^*$ in $\surf^*$ composed of four arcs $\gamma^*_1$,
 $\delta^*_1$, $\gamma^*_2$ and $\delta^*_2$. Since the orientations of
 $\gamma_1, \dots, \gamma_h$ are compatible, the arcs $\gamma_1^*$ and
 $\gamma_2^*$ are concurrently oriented as subarcs of $\gamma^*$, and they
 induce an orientation of $\gamma^*$ still compatible with $\gamma_3, \dots,
 \gamma_h$. 
 
 Similarly, we obtain an orientation on the new hole $\gamma'^*$ in
 $\surf^*_{0,h}$. We can also  extend $\zeta$ so that $\zeta(\gamma^*) =
 \zeta(\gamma'^*)$ (running along $\delta^*_1$ and $\delta^*_2$ with same
 speed). By induction, there is a homeomorphism $\bar \zeta^* \colon \surf^* \to
 \surf^*_{0,h}$, and the resulting $\bar \zeta$ is obtained by gluing $\surf^*$
 and $\surf^*_{0,h}$ back to $\surf$ and $\surf_{0,h}$.
\end{proof}

\begin{lemma}
\label{l:stand_rep}
Let $(\lambda_1, \dots, \lambda_s)$ be an independent system of handle-enclosing
curves in a surface $\surf$ of genus $g$ with $h$ holes, $s \leq g$. Let $\{\gamma_i\}_{i=1}^h$ be
the system of the boundary curves of the holes in $\surf$.
Then there is a homeomorphism $\psi \colon \surf \rightarrow \surf_{g,h}$ such that
$\psi(\gamma_i) = \partial H_i$, $i=1,2,\ldots,h$, and $\psi(\lambda_i) =
\partial D_i$, $i=1,2,\ldots,s$. Moreover, $\psi$ can be prescribed
on the $\gamma_i$, assuming that it preserves compatible orientations.
\end{lemma}

\begin{proof}
First we remark that we can assume that $s = g$. If $s < g$, we can 
extend $(\lambda_1, \dots, \lambda_s)$ to an independent system of
handle-enclosing of size $g$: We cut away each torus with hole  
$\surf^-_{\langle \lambda_i \rangle}$, obtaining a surface of genus $g-s$ homeomorphic to
$\surf_{g-s,h+s}$. Then we can find an independent system of $g-s$
handle-enclosing loops  in this surface. In sequel, we assume 
$s = g$.

  Let us cut $\surf$ along the curves $\lambda_1,\ldots,\lambda_s$.
 It decomposes into
a collection $T_1,\ldots,T_g$, where each $T_i$ is a torus with hole
 (with $\partial T_i=\lambda_i$),
and one planar surface $\surff$ with $g + h$ holes 
(the boundary curves of $\surff$ are the
$\lambda_i$ and the $\gamma_i$). In particular, $\surf$
decomposes into the same collection of surfaces (up to a homeomorphism) as
$\surf_{g,h}$ when cut along~$\partial D_i$. Let $\surff'$ be the planar
surface in this decomposition of $\surf_{g,h}$.

As we assume in the lemma, $\psi$ can be prescribed on some closed curves of $\partial
\surff$ while preserving compatible orientations. It can also be extended
so that it maps each  $\lambda_i$ to $\partial D_i$,
 while preserving compatible orientations
between $\surff$ and $\surff'$. Then
we have, by Lemma~\ref{l:planar_from_boundary}, a homeomorphism between
$\surff$ and $\surff'$ extending $\psi$. 

Finally, this homeomorphism can  also be extended to all the
$T_i$, one by one. Note
that preserving the orientations is not an issue in this case since the
torus with hole admits an automorphism reversing the orientation of the
boundary curve.
\end{proof}

We state the following corollary of Lemma~\ref{l:stand_rep}, which will be useful 
in Section~\ref{s:non_or}.

\begin{corollary}
\label{c:comp_ori}
  Let $\surf_1$ and $\surf_2$ be two orientable surfaces of genus $g$ with $h$
  holes. Let $\zeta \colon \partial \surf_1 \to \partial \surf_2$ be a
  homeomorphism of the boundaries that preserves compatible orientations. Then
  $\zeta$ extends to a homeomorphism  $\psi$ of $\surf_1$ and $\surf_2$. 
\end{corollary}

\begin{proof}
  We find an arbitrary homeomorphism $\zeta_1 \colon \partial \surf_1 \to
  \partial \surf_{g,h}$
  that preserves compatible orientations. Then the homeomorphism
  $\zeta_2\colon \partial \surf_2  \to \partial \surf_{g,h}$ defined as $\zeta_2 = \zeta_1 
  \zeta^{-1}$ preserves compatible orientations as well. Using
  Lemma~\ref{l:stand_rep} (with $s = 0$), we obtain extensions $\psi_1\colon
  \surf_1 \to \surf_{g,h}$ and $\psi_2\colon\surf_2 \to \surf_{g,h}$. Then $\psi
  := \psi^{-1}_2  \psi_1$ is the required homeomorphism.
\end{proof}

\begin{lemma}
\label{l:cut_S}
Let $\surf$ be a surface of genus $g$ with $h$ holes.
Let $(\delta_1, \dots, \delta_n)$ be an almost disjoint system of curves on $\surf$. 
Then there is an
independent system of $s \geq g - n$ handle-enclosing curves $\lambda_1, \dots,
\lambda_s$ such that each of the tori with hole
 $\surf^-_{\langle \lambda_j \rangle}$ 
is disjoint from $\bigcup_{i=1}^n \delta_i$.
\end{lemma}

\begin{proof}
Let us cut $\surf$ along $\{\delta_i\}_{i=1}^n$ obtaining several components $\surf_1,
\dots, \surf_q$. If we cut along the curves one by one, we see that 
Lemma~\ref{l:genera}(a) implies
$$
g(\surf_1) + \cdots + g(\surf_q) \geq g(\surf) - n. 
$$
In each $\surf_k$ we find an independent system of $g(\surf_k)$ handle-enclosing
curves (this can be done by transforming $\surf_k$ 
into the standard representation).  
The union of these independent systems yields
a system as in the lemma.
\end{proof}

\begin{proof}[Proof of Proposition~\ref{p:bound_g}(i)]
Let $\surf$ be a surface of genus $g$ with $h$ holes. Let
$A = (\alpha_1,\ldots,\alpha_m)$ and $B = (\beta_1,\ldots,\beta_n)$ be two
almost disjoint systems of
curves in $\surf$. 

Our task is to find a $\partial$-automorphism $\varphi$ of $\surf$
such that the number of crossings between $\alpha_1, \dots, \alpha_m$ and
$\varphi(\beta_1), \dots, \varphi(\beta_n)$ is at most $f_{g-s,h+s}(m,n)$,
 where  $s := \min(g - m, g - n)$. 
(Let us recall that we assume that $g > m,n$, and therefore $s > 0$.)

By Lemma~\ref{l:cut_S} there is an independent
system of handle-enclosing curves $\lambda_{1,\alpha}, \dots,
\lambda_{s,\alpha}$ such that the corresponding
tori with hole are disjoint from the curves in $A$. 
Consequently, by Lemma~\ref{l:stand_rep}, we have a homeomorphism 
$\psi_{\alpha}\colon \surf \rightarrow \surf_{g,h}$, extending a fixed 
homeomorphism $\psi'\colon \partial \surf \rightarrow \partial
\surf_{g,h}$, which preserves compatible orientations and maps each
$\lambda_{k, \alpha}$ to $\partial D_k$ (using the notation
from the definition of a standard representation).

Similarly, we have an independent system of handle-enclosing curves
$\lambda_{1,\beta}, \dots,
\lambda_{s,\beta}$ with the corresponding
tori with hole disjoint from the curves in $B$. We also have a
homeomorphism $\psi_{\beta}\colon \surf \rightarrow \surf_{g,h}$ extending $\psi'$
that maps the (closed) curves $\lambda_{k, \beta}$ to~$\partial D_k$.

Now we have two systems $A' = (\psi_\alpha(\alpha_1), \dots,
\psi_\alpha(\alpha_m))$ and $B' = (\psi_\beta(\beta_1), \dots,
\psi_\beta(\beta_m))$ of curves in $\surf_{g,h}$ avoiding the
tori with hole bounded by
the $\partial D_i$. Let us remove these tori (only for $i \leq s$)
obtaining a new surface $\surf^*$ of genus $g -s$ with $h + s$ holes. We find a
$\partial$-automorphism $\varphi^*$ of $\surf^*$ such that number 
of intersections
between $A'$ and $\varphi^*$-images of the curves in $B'$ is at most
$f_{g-s,h+s}(m,n)$. Since $\varphi^*$ fixes the boundary, 
it can be extended to a $\partial$-automorphism  $\varphi_{g,h}$ of 
$\surf_{g,h}$ while introducing no
new intersections. Finally, $\varphi := \psi^{-1}_{\alpha} \varphi_{g,h}
\psi_\beta$ is the required $\partial$-automorphism of~$\surf$.
\end{proof}

\subsection{Nonorientable Surfaces}
\label{subs:nonorientable}

The proof of Proposition~\ref{p:bound_g_nonorientable} is similar
to the previous proof but simpler, since we need not worry about
orientations.

\begin{lemma}
\label{l:homeo_no}
   Let $\nsurf$ and $\nsurf'$ be two nonorientable surfaces with the
 same genus and number of holes. Let $\psi_0\colon \partial \nsurf \to \partial
   \nsurf'$ be a homeomorphism of the boundaries. Then $\psi_0$ extends
   to a homeomorphism $\psi\colon \nsurf \to \nsurf'$.
\end{lemma}

\begin{proof}
  By the classification of surfaces, $\nsurf$ and $\nsurf'$ are homeomorphic. Given
  two boundary components, there is a self-homeomorphism of $\nsurf$ that
  exchanges these components.
 Therefore, we know that there is a homeomorphism $\psi_1\colon \nsurf \to
 \nsurf$ such that for each component $C$ of $\partial \nsurf$ the images
 $\psi_0(C)$ and $\psi_1(C)$ coincide (as sets). However, if we equip $C$ with
 an orientation, it might happen that $\psi_0(C)$ and $\psi_1(C)$ have opposite
 orientations. In such case,  we consider a self-homeomorphism
 $\psi_C$ of $\nsurf$ that reverts the orientation of $C$ and fixes all
 other boundary components. Here is an example of such a self-homeomorphism:
\immfig{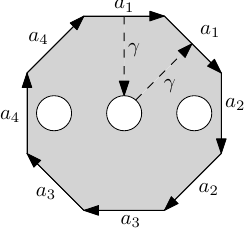}
Up to a homeomorphism, we can consider $\nsurf$ as a polygon with
 holes whose edges are identified according to the labels.
  By moving the middle hole along $\gamma$, we revert its orientation
without affecting the other holes.

By gradually composing $\psi_1$ with the $\psi_C$ for
 those $C$ on which orientations disagree, we can get a self-homeomorphism 
 of $\nsurf$ such that $\psi_0(C)$ and $\psi_2(C)$ have
 compatible orientations for every $C$. Finally, by a local modification of
 $\psi_2$ at small neighborhood of every $C$ we can get a self-homeomorphism
 $\psi$ of $\nsurf$ that agrees with $\psi_0$ on $\partial \nsurf$.
\end{proof}

Similar to the orientable case, we will use a certain canonical
representation $\nsurf_{g,h}$ for a nonorientable surface of genus $g$ with
$h$ holes. We recall that a \emph{cross-cap} in a nonorientable surface $\nsurf$ is a subset of
$\nsurf$ which is homeomorphic to a M\"{o}bius band. Note that the boundary of a
cross-cap is a single closed curve. A standard way of 
representing a nonorientable surface of genus $g$ with $h$ holes is to
remove $h$ disjoint disks from the $2$-sphere and replace other $g$ disjoint
disks with cross-caps. However, here it is more convenient to 
replace all but at most two of the cross-caps by handles:
indeed, for $g \geq 3$, a pair of
cross-caps can be replaced with a handle 
(this is sometimes called \emph{Dyck's Theorem},
see, e.g., \cite[Lemma~3]{FrancisWeeks:ConwayZIPProof-1999}; note that it is essential
that at least one cross-cap remain present).

Thus, we can define a \emph{convenient} representation (as opposed to
the standard one mentioned above) $\nsurf_{g,h}$ as follows.
We again start with the sphere $S^2$, and we remove $h$ pairwise
disjoint disks $H_1, \dots, H_h$. Then we remove
$\lfloor(g-1)/2\rfloor$ more disjoint
disks $D_1, \dots, D_{\lfloor(g-1)/2\rfloor}$ and
attach a torus with hole along boundary of each $D_i$. Finally, we 
remove one (for $g$ odd) or two (for $g$ even) extra disks
and we attach M\"obius bands along these disks.
 Here is the convenient representation of $\nsurf_{6,2}$:
\immfig{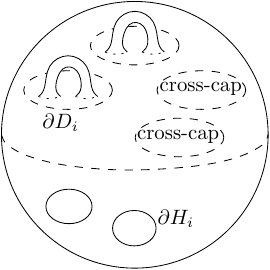}
%
%

\begin{lemma}
\label{l:stand_rep_no}
Let $(\lambda_1, \dots, \lambda_s)$ be an independent system of handle-enclosing
curves in a nonorientable surface $\nsurf$ of genus $g$ with $h$ holes, $s \leq
\lfloor(g-1)/2 \rfloor$. Let $\{\gamma_i\}_{i=1}^h$ be
the system of the boundary curves of the holes in $\nsurf$.
Then there is a homeomorphism $\psi \colon \nsurf \rightarrow \nsurf_{g,h}$ such that
$\psi(\gamma_i) = \partial H_i$, $i=1,2,\ldots,h$, and $\psi(\lambda_i) =
\partial D_i$, $i=1,2,\ldots,s$. Moreover, $\psi$ can be prescribed
on the $\gamma_i$.
\end{lemma}

\begin{proof}
  The proof is analogous to that of Lemma~\ref{l:stand_rep}.
  Let us cut $\nsurf$ along the curves $\lambda_1,\ldots,\lambda_g$.
 It decomposes into
a collection $T_1,\ldots,T_s$, where each $T_i$ is a torus with hole
 (with $\partial T_i=\lambda_i$),
and one nonorientable surface $\hat \nsurf$ of genus $g - 2s$ with $h + s$ holes
(the boundary curves of $\nsurf$ are the
$\lambda_i$ and the $\gamma_i$). In particular, $\nsurf$
decomposes into the same collection of surfaces (up to a homeomorphism) as
$\nsurf_{g,h}$ when cut along the $\partial D_i$. Let $\nsurf'$ be
the nonorientable surface in the decomposition of $\nsurf_{g,h}$.

By Lemma~\ref{l:homeo_no}, we have a homeomorphism
between $\hat \nsurf$ and $\nsurf'$ extending a given
homeomorphism of the boundary curves.
This homeomorphism can be also extended to all $T_i$, one by one. 
\end{proof}

\begin{lemma}
\label{l:cut_S_no}
Let $\nsurf$ be a nonorientable surface of genus $g$ with $h$ holes.
Let $(\delta_1, \dots, \delta_n)$ be an almost disjoint system of curves on $\surf$. 
Then there is an
independent system of $s \geq g/2 - 2n -1$ handle-enclosing curves $\lambda_1, \dots,
\lambda_s$ such that each of the tori with hole
 $\surf^-_{\langle \lambda_j \rangle}$ 
is disjoint from $\bigcup_{i=1}^n \delta_i$.
\end{lemma}

\begin{proof}
Let us cut $\nsurf$ along $\{\delta_i\}_{i=1}^n$, obtaining several components
$\surf_1,
\dots, \surf_q$, $q \leq n + 1$ (some of them may be orientable
and some nonorientable). 
Cutting along the curves one by one, 
we see that Lemma~\ref{l:genera}(b) implies
$$
g_e(\surf_1) + \cdots + g_e(\surf_q) \geq g_e(\surf) - 2n. 
$$
In each $\surf_k$ we find an independent system of at least $(g_e(\surf_k) -
2)/2$ handle-enclosing curves. Indeed, if $\surf_k$ is orientable, then we can
find even $g_e(\surf_k)/2$ such curves by transforming $\surf_k$ to the standard
representation. If $\surf_k$ is nonorientable, then we find at least
$(g_e(\surf_k) - 2)/2$ such curves by transforming $\surf_k$ to the convenient
representation.

The union of these independent systems yields
a system as in the lemma (using $g = g_e(\surf)$ and $q
\leq n+ 1$).
\end{proof}

\begin{proof}[Proof of Proposition~\ref{p:bound_g_nonorientable}]
  The proof is now almost same as for Proposition~\ref{p:bound_g}(i).

Let $\nsurf$ be a nonorientable surface of genus $g$ with $h$ holes. Let
$A = (\alpha_1,\ldots,\alpha_m)$ and $B = (\beta_1,\ldots,\beta_n)$ be two
almost disjoint systems of
curves in $\nsurf$. 

Our task is to find a $\partial$-automorphism $\varphi$ of $\nsurf$
such that the number of crossings between $\alpha_1, \dots, \alpha_m$ and
$\varphi(\beta_1), \dots, \varphi(\beta_n)$ is at most $\hat f_{g-2s,h+s}(m,n)$,
 where  $s := \min(\lceil g/2\rceil - 2m - 1, \lceil g/2 \rceil - 2n - 1)$. 
Note that $g - 2s = 4\maxmn + 2 - (g \smod 2)$ and $h + s = h
+ \lceil g/2 \rceil - 2\maxmn - 1$ as required ($\maxmn=\max(m,n)$). 
 (Let us also recall that we assume that $g > 4\maxmn + 2$, and so $s > 0$.)

By Lemma~\ref{l:cut_S_no} there is an independent
system of handle-enclosing curves $\lambda_{1,\alpha}, \dots,
\lambda_{s,\alpha}$ such that the corresponding
tori with hole are disjoint from the curves in $A$. 
Consequently, by Lemma~\ref{l:stand_rep_no}, we have a homeomorphism 
$\psi_{\alpha}\colon \nsurf \rightarrow \nsurf_{g,h}$, extending a fixed 
homeomorphism $\psi'\colon \partial \nsurf \rightarrow \partial
\nsurf_{g,h}$, which maps each
 $\lambda_{k, \alpha}$ to $\partial D_k$.

Similarly, we have an independent system of handle-enclosing curves
$\lambda_{1,\beta}, \dots,
\lambda_{s,\beta}$ with the corresponding
tori with hole disjoint from the curves in $B$. We also have a
homeomorphism $\psi_{\beta}\colon \nsurf \rightarrow \nsurf_{g,h}$ extending $\psi'$
that maps the each $\lambda_{k, \beta}$ to~$\partial D_k$.

Now we have two systems $A' = (\psi_\alpha(\alpha_1), \dots,
\psi_\alpha(\alpha_m))$ and $B' = (\psi_\beta(\beta_1), \dots,
\psi_\beta(\beta_m))$ of curves in $\nsurf_{g,h}$ avoiding the
tori with hole bounded by
the $\partial D_i$. Let us remove these tori (only for $i \leq s$)
obtaining a new surface $\nsurf^*$ of genus $g -2s$ with $h + s$ holes. We find a
$\partial$-automorphism $\varphi^*$ of $\nsurf^*$ such that number 
of intersections
between $A'$ and $\varphi^*$-images of the curves in $B'$ is at most
$f_{g-s,h+s}(m,n)$. Since $\varphi^*$ fixes the boundary, 
it can be extended to a $\partial$-automorphism  $\varphi_{g,h}$ of 
$\nsurf_{g,h}$ while introducing no
new intersections. Finally, $\varphi := \psi^{-1}_{\alpha} \varphi_{g,h}
\psi_\beta$ is the required $\partial$-automorphism of~$\nsurf$.
\end{proof}

\section{Reducing the Orientable Genus to \boldmath$0$}
\label{s:redu2}


Here we prove Proposition~\ref{p:bound_g}(ii). We start with
some preliminaries.

Let $g \geq 1$ and let $M_g$ be a $4g$-gon with edges consecutively labeled 
$a_1^+$,  $b_1^+$, $a^-_1$, $b^-_1$, $a_2^+$, $b_2^+$, $a_2^-$,\dots, $b^-_g$. 
The edges are
oriented: the $a^+_i$ and $b^+_i$ clockwise, and  the $a^-_i$ and $b^-_i$
counter-clockwise. By identifying the edges
$a^+_i$ and $a^-_i$, as well as $b^+_i$ and $b^-_i$, according to their
orientations, we
obtain an orientable surface $\surf_g$ of genus $g$. The polygon $M_g$ is a
\emph{canonical polygonal schema} for~$\surf_g$. 

Removing the interior of $M_g$, we obtain a system of $2g$ \emph{loops} (closed
curves
with distinguished endpoints), all having the same endpoint. 
This system of loops is a \emph{canonical system of loops} for $\surf_g$. 
The loop in $\surf_g$ obtained by identifying $a^+_i$ and $a^-_i$ is denoted 
by $a_i$. Similarly, we have the loops $b_i$. In the sequel, we 
assume that an orientable surface $\surf$ is given and
we look for a canonical system of loops induced by some canonical polygonal
schema; here is an example with the double-torus: 
\immfig{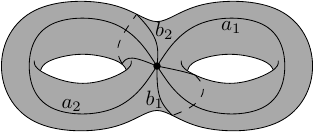}


Given a surface $\surf$ with boundary, we can extend the definition of canonical
system of loops for $\surf$ in the following way. We contract each boundary hole of
$\surf$ obtaining a surface $\tilde \surf$ without boundary. A system of loops $(a_1,
b_1, a_2, \dots, b_g)$ in $\surf$ is a \emph{canonical system of loops} for $\surf$ if
no loop intersects the boundary of $\surf$ and the resulting system $(\tilde a_1,
\tilde b_1, \tilde a_2 \dots, \tilde b_g)$ after the contractions is a
canonical system of loops for $\tilde \surf$.


\begin{lemma}
\label{l:can_hom}
Let $\lsys = (a_1, b_1, \dots, b_g)$ and $\lsys' = (a'_1, b'_1, \dots, 
b'_g)$ be two canonical systems of loops for a given orientable surface $\surf$
with or without boundary. Then, there is a $\partial$-automorphism $\psi$ of
$\surf$ transforming $\lsys$ to $\lsys'$ (it may not keep\footnote{It can be
seen from the proof that the labels are either kept or $\psi$ transforms $(a_1,
b_1, \dots, b_g)$ to $(b'_g, a'_g, \dots, 
a'_1)$.} the labels; that is,
$a_1$ need not be transformed to $a'_1$, etc.). 
\end{lemma}

\begin{proof}
If $\surf$ has no boundary, then the lemma immediately follows from the
definitions; $a_i$ is mapped to $a'_i$ and $b_i$ to $b'_i$.

If $\surf$ has a boundary, we first contract each of the holes, obtaining a surface
$\tilde \surf$. In particular, each hole $H_i$ becomes a point $h_i$. 
Let $\tilde \lsys$ and $\tilde \lsys'$ be the resulting canonical systems on
$\tilde \surf$. We find an automorphism $\tilde \psi_1$ of $\tilde \surf$
transforming $\tilde \lsys$ to $\tilde \lsys'$. 

The automorphism $\tilde \psi_1$ may or may not be orientation-preserving. If
$\tilde \psi_1$ preserves the orientation of $\tilde \surf$, 
we set $\tilde \psi_2 := \tilde \psi_1$. If $\tilde \psi_1$ 
reverts the orientation we set $\tilde \psi_2 := \tilde \psi_1  \tilde
\psi$ where $\tilde \psi$ is an orientation-reversing 
automorphism of $\tilde \surf$ transforming $\lsys$ to $\lsys$; see
Fig.~\ref{f:revert_ori}. In any case,
 $\tilde \psi_2$ preserves the orientation and
maps $\lsys$ to $\lsys'$.

\begin{figure}
\begin{center}
\includegraphics{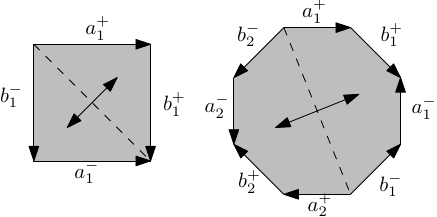}
\caption{Two examples of an automorphism $\tilde \psi$ reverting the
orientation of $\tilde \surf$ induced by a reflection of $M_g$ along one of the
diagonals. In general, $(a_1, b_1, a_2, \dots, b_g)$
is mapped to $(b_g, a_g, b_{g-1}, \dots, a_1)$.}
\label{f:revert_ori}
\end{center}
\end{figure}

We adjust $\tilde \psi_2$ to fix each 
$h_i$
(this is possible since $\tilde \surf$ remains connected after cutting
along $\tilde \lsys'$ and also since the  
points $h_i$ are disjoint from
the loops of $\tilde \lsys$). Then we decontract the points $h_i$ back to
holes, obtaining $\surf$. After this $\tilde \psi_2$ induces the required
$\partial$-automorphism $\psi$ of~$\surf$. (The obvious automorphism of $\surf$
obtained by decontraction of the holes need not fix boundary; however, it
can easily be modified to fix the boundary since $\psi_2$ preserves the
orientation.)
\end{proof}



We need a theorem of Lazarus et al.~\cite{Lazarus-al} in the following version.

\begin{theorem}[{cf.~\cite[Theorem 1]{Lazarus-al}}]
\label{t:L_al}
  Let $\surf$ be a triangulated surface without boundary 
  with total of $n$ vertices, edges and
  triangles. Then there is a canonical system of loops for $\surf$ avoiding the
  vertices of $\surf$ and meeting edges of $\surf$ at a finite number of points
  such that each loop of the system has at most $O(n)$ intersections with 
  the edges of the triangulation.
\end{theorem}


As we already mentioned in the introduction, the result is
essentially due to Vegter and Yap~\cite{VegterYap}. Lazarus et al. 
provide more details (\cite{VegterYap} is only an extended abstract),
and they have a slightly different representation for the
canonical system of loops, which is more convenient for our purposes.

From Theorem~\ref{t:L_al} we easily derive the following extension.

\begin{proposition}
\label{p:canonical}
  Let $\surf$ be an orientable surface of genus $g$ with or without boundary. Let
  $D = (\delta_1, \dots, \delta_n)$ be an almost disjoint 
system of curves on $\surf$.
  Then there is a canonical system of loops $\lsys = (a_1, b_1, \dots, b_g)$
  such that $D$ and $\lsys$ have $O(gn + g^2)$ crossings.
\end{proposition}

For the proof, we need the following lemma, which may very well be
folklore, but which we
haven't managed to find in the literature.

\begin{lemma}
\label{l:O(g+n)}
Let $G$ be a nonempty graph with at most $n$ vertices and
edges, possibly with loops and/or multiple edges, embedded in an
orientable surface $\surf$ of genus $g$ without boundary.
Then there is a graph $G'$ without loops or
multiple edges and with $O(g+n)$ vertices
and edges that contains a subdivision of $G$ and triangulates~$\surf$.
\end{lemma}

In the proof below we did not attempt to optimize the 
constant in the $O$-notation.
We thank Robin Thomas for a suggestion that helped us to simplify the proof.

\begin{proof}
We can assume that every vertex is connected to at least one edge;
if not, we add loops.

Let us cut $\surf$ along the edges of $G$. 
We obtain several components $\surf_1,
\dots, \surf_q$. By Lemma~\ref{l:genera} we know that 
$$
 g(\surf_1) + \cdots + g(\surf_q) \leq g.
$$
First, whenever $g(\surf_i) > 0$ for some $i$, we introduce a
canonical system of loops inside $g(\surf_i)$. For this we need one vertex and
$2g(\surf_i)$ edges, which gives at most $3g$
new vertices and edges in total. 
In this way we obtain a graph $G^1$ (containing $G$). 

We cut $\surf$ along the edges of $G^1$; the resulting components are all
planar. Inside each component $\surf^1_i$  we introduce a new vertex $v$ and
connect it to all vertices on the boundary of $\surf^1_i$; $v$ can be
connected to some boundary vertex $u$ by multiple edges if $u$ occurs on the
boundary of $\surf^1_i$ in multiple copies. This is easily achievable if we
consider, up to a homeomorphism, $\surf^1_i$ as a polygon, possibly with
tiny holes inside; see the left picture:
\immfig{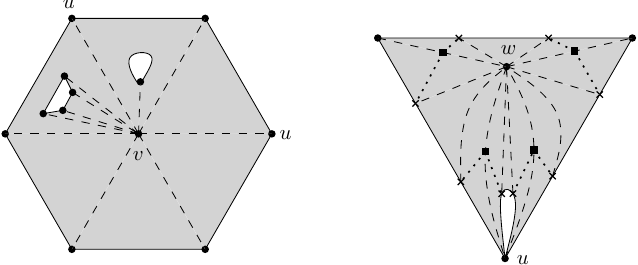} 
Since we have added at most $\deg u$ edges per vertex $u$ of $G^1$,
we obtain a graph $G^2$, still with $O(g+n)$ vertices and edges.


We cut $\surf$ along the edges of $G^2$. The resulting components $\surf^2_i$
are all planar and, in addition, they have a single boundary curve. 
We subdivide each edge of $G^2$ twice, we introduce a new vertex $w$ 
in each $\surf^2_i$, and we connect $w$ to all vertices on the boundary of 
$\surf^2_i$ (including the vertices obtained from the subdivision). 
If $w$ is connected to a vertex $u$ of $G^2$ on the boundary of $\surf^2_i$, 
we further subdivide the edge $uw$ and we connect the newly introduced 
vertex to the two neighbors of $u$ along the boundary of $\surf^2_i$; 
this is illustrated in the right picture above.

This yields the required graph $G'$. Indeed, we have subdivided all
loops and multiple edges in $G^2$, and we do not introduce any 
new loops or multiple edges (because of the subdivision of $uw$ edges). 
Each face of $G'$ is triangular;
therefore, we have a triangulation. 
The size of $G'$ is bounded by $O(g+n)$.
\end{proof}

\begin{proof}[Proof of Proposition~\ref{p:canonical}]
  If $\surf$ contains
  holes,  we contract them, find the canonical system on the contracted
  surface, and decontract the holes  (without affecting the number of
  crossings). Thus, we can assume that $\surf$ has no boundary.

  Now we form a graph $G$ embedded in $\surf$ in the following way. The vertex set
  of $G$ contains all endpoints of arcs in $D$. For a closed curve in $D$, we pick
  a vertex on the curve.
  Each arc of $D$ induces an edge in $G$. Each closed curve of $D$ induces a loop in
  $G$. This finishes the construction of~$G$.

  The graph $G$ has $O(n)$ vertices and edges. Let $G'$ be the graph from
  Lemma~\ref{l:O(g+n)} containing a subdivision of $G$.

%
%
%
Now we can use Theorem~\ref{t:L_al} for the triangulation
  given by $G'$ to obtain the required canonical system of loops.
\end{proof}

\begin{proof}[Proof of Proposition~\ref{p:bound_g}(ii)]
Let $\surf$ be a surface of genus $g$ with $h$ holes. Let $A = (\alpha_1,
\dots, \alpha_m)$ and $B = (\beta_1, \dots, \beta_n)$ be two almost-disjoint
systems of curves. Our task is to find a $\partial$-automorphism 
$\varphi$ of $\surf$
such that $\alpha_1, \dots, \alpha_m$ and $\varphi(\beta_1), \dots, 
\varphi(\beta_m)$ 
have at most $f_{0,h+1}(m',n')$ intersections,
 where $m' \leq cg(m+g)$ and $n'
\leq cg(n+g)$ for some constant $c$. Proposition~\ref{p:bound_g}(ii)
 then follows from the monotonicity of $f_{g,h}(m,n)$ in $m$ and $n$.

 Let $\lsys_\alpha$ be a canonical system of loops as in
 Proposition~\ref{p:canonical} used with $(\alpha_1, \dots, \alpha_m)$, and
let  $\lsys_\beta$ be a canonical system of loops  as in
  Proposition~\ref{p:canonical} used with $(\beta_1, \dots, \beta_n)$.
 
  According to Lemma~\ref{l:can_hom}, there is a 
$\partial$-automorphism $\psi$ of $\surf$ transforming $\lsys_\beta$ 
to $\lsys_\alpha$. This homeomorphism induces a new system of 
curves $B_{\psi} := (\psi(\beta_1), \dots, \psi(\beta_n))$.

We cut $\surf$ along $\lsys_\alpha$, obtaining a new, planar
 surface $\surf'$ with $h + 1$ holes
(one new hole appears along the cut). 
According to the
choice of $\lsys_\alpha$ and $\lsys_\beta$, 
the systems $A$ and $\lsys_\alpha$ have at most $O(gm + g^2)$ intersections.
Similarly, $B_\psi$ and $\lsys_\alpha$ have at most $O(gn + g^2)$
intersections. Thus, $A$ induces a system $A'$ of $m' \leq cg(m + g)$ 
new curves on 
$\surf'$, and $B_\psi$ induces a system $B'$ of $n' \leq cg(n + g)$ 
new curves on $\surf'$.
From the definition of $f$, we find a $\partial$-automorphism $\varphi'$
of $\surf'$  such that $A'$ has
at most $f_{0,h+1}(m', n')$ intersections with $\varphi'(B')$.
Then we glue $\surf'$ 
back to $\surf$, inducing the required $\partial$-automorphism $\varphi$ of $\surf$.
\end{proof}

\section{Reducing the Nonorientable Case to the Orientable One}
\label{s:non_or}
In this section, we prove Proposition~\ref{p:or_red}.

Let $\nsurf$ be a nonorientable surface with $h\geq 0$ holes and nonorientable genus $g\geq 1$.

Our approach to prove Proposition~\ref{p:or_red} is similar in spirit to the
proof of Proposition~\ref{p:bound_g}~(ii). The difference is that instead of
cutting an orientable surface along a canonical system of loops to get a planar one,
we cut the nonorientable surface $\nsurf$ along one distinguished closed curve 
so as to obtain an orientable surface.

We recall that, given a closed curve $\lambda$ on a surface $\nsurf$, 
the surface obtained by cutting
$\nsurf$ along $\lambda$ is denoted by $\cut{\nsurf}{\lambda}$.

Formally, an \emph{orientation-enabling curve} in a nonorientable surface $\nsurf$ is
a properly embedded closed curve $\lambda$ such that 
$\cut{\nsurf}{\lambda}$ is orientable. It follows that an orientation-enabling 
curve is non-separating,
 since attaching two orientable components along a closed curve 
yields an orientable surface. 

It is not hard to see that any nonorientable surface admits an
orientation-enabling curve; 
it can be explicitly found in the convenient representation of the surface introduced in Section~\ref{subs:nonorientable}.
For technical reasons, however, we will need to find an orientation-enabling
curve $\lambda$ 
that also satisfies two additional properties: $\lambda$ should be compatible with orientations
of the boundary curves of the holes in the surface (in a sense to be made precise below), and 
it should also be compatible with a given system $D$ of curves on $\nsurf$, in the sense that 
we can bound the number of intersections between $\lambda$ and $D$.

The first ingredient for the proof of Proposition~\ref{p:or_red} is an analogue of Lemma~\ref{l:can_hom}.
A perfect analogue would be to show that any two orientation-enabling curves of $\nsurf$ 
can be transformed into one another by a $\partial$-automorphism of $\nsurf$. 
However, it turns out that for nonorientable surfaces with holes
this is not true in general; see Example~\ref{e:many_io} 
below. For this reason, we need the requirement of compatible orientations in the following lemma.

\begin{lemma}
\label{l:oi_hom} Let $\nsurf$ be a nonorientable surface with boundary curves $\gamma_1, \dots,
\gamma_h$ and let $\lambda$ and $\kappa$ be two orientation-enabling curves in $\nsurf$.
Suppose that we have chosen orientations each of the curves $\gamma_1, \dots,
\gamma_h$ and for $\lambda$ and $\kappa$.

Supposed furthermore that the induced orientations of the boundary curves of $\cut{\nsurf}{\lambda}$ 
are mutually compatible, in the sense explained before Lemma~\ref{l:planar_from_boundary}, and that the same holds
for the boundary curves of $\cut{\nsurf}{\kappa}$ (we stress that the compatibility condition also applies to the boundary
 curves originating from $\lambda$ and $\kappa$, respectively).

Then there is a $\partial$-automorphism $\psi$ of
$\nsurf$ transforming $\lambda$ to $\kappa$. 
\end{lemma}

The second ingredient for the proof of Proposition~\ref{p:or_red} is the following existence result,
analogous to Proposition~\ref{p:canonical}.

\begin{proposition}
\label{p:oi_exists}
  Let $\nsurf$ be a nonorientable surface of genus $g$ with or without
  boundary. Let $\gamma_1, \dots,
  \gamma_h$ be the boundary curves of $\nsurf$ given with some orientations.
  Let $D = (\delta_1, \dots, \delta_n)$ be an almost disjoint 
system of curves on $\nsurf$.
  Then there is an orientation-enabling curve $\lambda$ such that $D$ and
  $\lambda$ have $O(g+n)$ crossings and such that $\lambda$ can be equipped with an
  orientation such that the induced orientations of the boundary curves on
  $\nsurf_{\langle \lambda \rangle}$ are mutually compatible.
\end{proposition}

Finally, we will need the following simple lemma that relates the genus and number of
holes of $\nsurf$ to the corresponding quantities for $\cut{\nsurf}{\lambda}$. 


\begin{lemma}
\label{l:cut_nonor}
Let $\nsurf$ be a nonorientable surface of genus $g$ with $h$ holes and let
$\lambda$ be an orientation-enabling curve. Let $g_\lambda$ be the
(orientable) genus of $\nsurf_{\langle \lambda \rangle}$ and $h_\lambda$ be the number of holes of
$\nsurf_{\langle \lambda \rangle}$.
\begin{itemize}
\item[\textup{(a)}]
 If $g$ is odd, then $\lambda$ is one-sided,
$g_\lambda = (g-1)/2$, and $h_\lambda = h + 1$.
\item[\textup{(b)}]
 If $g$ is even, then $\lambda$ is two-sided, $g_\lambda = (g-2)/2$, and $h_\lambda = h + 2$.
\end{itemize}
\end{lemma}

\begin{proof}
Let us recall that we have the following relations for the Euler characteristic:
$\chi(\nsurf) = 2 - g - h$ since $\nsurf$ is nonorientable, and
$\chi(\nsurf_{\langle \lambda \rangle}) = 2 - 2g_\lambda - h_\lambda$ since
$\nsurf_{\langle \lambda \rangle}$ 
is orientable.
We also have $\chi(\nsurf) = \chi(\nsurf_{\langle \lambda \rangle})$ 
since the Euler characteristic of the closed curve $\lambda$ is $0$. 

If $\lambda$ is one-sided, then $h_\lambda = h+1$,
 implying $g_\lambda = (g-1)/2$. In particular, $g$ must be odd.
If $\lambda$ is two-sided, then $h_\lambda = h +2$,
 implying $g_\lambda = (g-2)/2$. In particular, $g$ must be even.
This proves the lemma, since we have exhausted all possibilities.
\end{proof}

Now we are ready to prove Proposition~\ref{p:or_red}.

\begin{proof}[Proof of Proposition~\ref{p:or_red}.]
Let $\nsurf$ be a nonorientable surface of (nonorientable) genus $g$ with $h$ holes. Let $A = (\alpha_1,
\dots, \alpha_m)$ and $B = (\beta_1, \dots, \beta_n)$ be two almost-disjoint
systems of curves. Our task is to find a $\partial$-automorphism 
$\varphi$ of $\nsurf$
such that $\alpha_1, \dots, \alpha_m$ and $\varphi(\beta_1), \dots, 
\varphi(\beta_m)$ 
have at most $f_{g',h'}(c(g+m),c(g+n))$
intersections, with $g'=\lfloor (g-1)/2 \rfloor$ and 
$h'=h + 1 + (g \smod 2)$. 

Let us fix orientations of the boundary curves of $\nsurf$ arbitrarily.
 Let $\lambda_\alpha$ be an orientation-enabling curve obtained from  
 Proposition~\ref{p:oi_exists} applied to $\nsurf$ and the system $A=(\alpha_1, \dots, \alpha_m)$, and
let  $\lambda_\beta$ be an orientation-enabling curve obtained from 
Proposition~\ref{p:oi_exists} used for $\nsurf$ and the system $B=(\beta_1, \dots, \beta_n)$.
 
According to Lemma~\ref{l:oi_hom}, there is a 
$\partial$-automorphism $\psi$ of $\nsurf$ transforming $\lambda_\beta$ 
to $\lambda_\alpha$. This homeomorphism induces a new system of 
curves $B_{\psi} := (\psi(\beta_1), \dots, \psi(\beta_n))$.

We cut $\nsurf$ along $\lambda_\alpha$, obtaining a new, orientable
surface $\surf$. By Lemma~\ref{l:cut_nonor}, $\surf$ has genus $g'$
 and $h'$ holes. 
By the choice of $\lambda_\alpha$, the system $A$ and the (closed) curves $\lambda_\alpha$ have at most $O(m+g)$ intersections.
Similarly, by our choices of $\lambda_\beta$ and of $\psi$, the system $B_\psi$ and $\lambda_\alpha=\psi(\lambda_\beta)$ 
have at most $O(n + g)$ intersections. Thus, $A$ induces a system $A'$ of $m' \leq c(m + g)$ 
new curves on $\surf$, and $B_\psi$ induces a system $B'$ of $n' \leq c(n + g)$ 
new curves on $\surf$.
By the definition of $f$ and monotonicity, we find a $\partial$-automorphism $\varphi'$
of $\surf$  such that $A'$ has at most $f_{g',h'}(c(g+m),c(g+n))$ 
intersections with $\varphi'(B')$.

By the construction, $\varphi'$ is compatible with the operation 
of undoing the cutting of $\nsurf$ along $\lambda_\alpha$,
i.e., $\varphi'$ induces a $\partial$-automorphism $\varphi$ of $\nsurf$, and this $\varphi$ yields the desired bound on the 
entanglement number of $A$ and~$B$.
\end{proof}

\subsection{Uniqueness of Orientation-Enabling Curves}
In this section, we prove Lemma~\ref{l:oi_hom} (which is fairly easy, using the classification of surfaces).
First, however, we briefly digress to describe the promised example that explains why the compatibility assumptions 
in the lemma are necessary. (The reader may skip this example since
it is not used in any of the proofs.)

\begin{example}
\label{e:many_io}
  Let us consider a fixed nonorientable surface $\nsurf$; for concreteness,
  let us take the projective plane with $4$ holes. We assume that $\nsurf$ 
  is obtained by identifying antipodal points on the boundary of 
the disk with holes.
Let us consider orientation-enabling curves
  $\kappa$ and $\lambda$ as below: 
\immfig{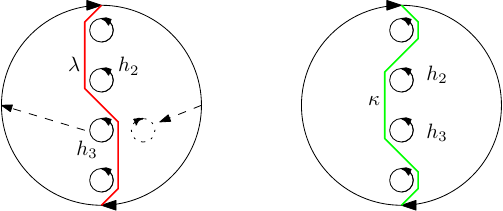}
  We want to show that there is no $\partial$-automorphism 
of $\nsurf$ transforming
  $\lambda$ to $\kappa$.

We see that the holes $h_2$ are (locally) on the same side of $\kappa$ whereas
they are on different sides of $\lambda$.
Let $\nsurf'$ be the surface
  obtained by gluing $h_2$ and $h_3$ according to the indicated orientations. If
  there is a $\partial$-automorphism transforming $\lambda$ to $\kappa$, then
  the surfaces $\nsurf'_{\langle \lambda \rangle}$ and $\nsurf'_{\langle \kappa
  \rangle}$ must be homeomorphic. However, $\nsurf'_{\langle \kappa \rangle}$
  is obtained from $\nsurf_{\langle \kappa \rangle}$ by introducing a
  \emph{cross-handle} (i.e., two cross-caps) since the orientations of $h_2$ and $h_3$ are compatible on 
  $\nsurf_{\langle \kappa \rangle}$, and thus $\nsurf'_{\langle \kappa
  \rangle}$ is a nonorientable surface. On the other hand, $\nsurf'_{\langle 
  \lambda \rangle}$ is obtained by introducing a handle 
  (think of moving $h_3$ as the arrow in the picture above
indicates). Therefore, $\nsurf'_{\langle \lambda
  \rangle}$ is orientable. We conclude that there is no $\partial$-automorphism
  of $\nsurf$ transforming $\lambda$ to $\kappa$.

  By this approach, if we have $h$ holes, we can construct $2^{h-1}$ different
  orientation-enabling curves with respect to $\partial$-homeomorphisms. (By
  an approach similar to the proof of Lemma~\ref{l:oi_hom}, one can actually
  see that there are exactly $2^{h-1}$ different orientation-enabling
  curves, but we will not need this in what follows.) 
\end{example}


We now proceed to provide the details for the proof of Lemma~\ref{l:oi_hom}. 

\begin{proof}[Proof of Lemma~\ref{l:oi_hom}]
  Both $\nsurf_{\langle \lambda \rangle}$ and $\nsurf_{\langle \kappa \rangle}$ have the same number of holes and
  same genus according to Lemma~\ref{l:cut_nonor}, and so they
  are homeomorphic. The idea is that a homeomorphism $\psi'$ of
  $\nsurf_{\langle \lambda \rangle}$ and $\nsurf_{\langle \kappa \rangle}$ induces the required
$\partial$-automorphism $\psi$ of $\nsurf$ simply by undoing the operations of cutting $\nsurf$ along
$\lambda$ and $\kappa$, respectively. 
We need to be little careful, however, and to check that $\psi'$ preserves the boundary and is 
compatible with the gluing.

Let $B_\lambda$ be the part of the boundary of $\nsurf_{\langle \lambda \rangle}$ obtained from
$\lambda$ when cutting $\nsurf$. According to Lemma~\ref{l:cut_nonor},
$B_\lambda$ consists of one or two closed curves, depending of the parity of $g$. We define $B_\kappa$ analogously.
We have an involution $i_{\lambda}$ on $B_\lambda$ such that the identification
of all pairs $x$ and $i_\lambda(x)$ yields $\nsurf$. We have an analogous
involution $i_{\kappa}$ on $B_\kappa$. We need a homeomorphism
$\psi'\colon \nsurf_{\langle \lambda \rangle} \to \nsurf_{\langle \kappa
\rangle}$ that is compatible with these
involutions (that is, $\psi'  i_\lambda = i_\kappa \psi'$ on
$B_\lambda$), so that gluing back induces an automorphism of
$\nsurf$. We also need that $\psi'$ fixes the other holes so that
we obtain a $\partial$-automorphism.

We can define $\psi'$ first on $\partial \nsurf_{\langle \lambda \rangle}$ so that the requirements
above are satisfied. Due to our compatibility assumptions, we can use
Corollary~\ref{c:comp_ori} to get $\psi'$ on the whole $\nsurf_{\langle \lambda
\rangle}$. As we have
already mentioned, we obtain the required $\psi$ by gluing back
$\nsurf_{\langle \lambda \rangle}$
and $\nsurf_{\langle \kappa \rangle}$ to $\nsurf$.
\end{proof}

\subsection{Existence of Orientation-Enabling Curves}

In this section, we prove Proposition~\ref{p:oi_exists}.

The proof will be subdivided into several steps.
As in the proof of Proposition~\ref{p:bound_g}~(ii),
we will replace the given system $D$ of curves by a suitable
triangulation of the surface and show that there exists an
orientation-enabling curve $\lambda$ in $\nsurf$ that
intersects the edges of the triangulation in a controlled way. 
We will look for $\lambda$ by choosing local orientations of
the triangles of (a suitable refinement of) the given triangulation of $\nsurf$. 
Then $\lambda$ will appear as the ``ceasefire line''
 where the local orientations disagree. This will
automatically guarantee that the surface $\cut\nsurf\lambda$ 
obtained after cutting along $\lambda$ is orientable. However, we still have to
argue that we can choose the local orientations so that
$\lambda$ is a single closed curve, and so that
it does not intersect the original
triangulation too often. Below we provide the details.

\heading{Local orientations.} Let us assume that $\nsurf$ is a triangulated
surface. We equip each triangle with a \emph{local orientation} (which can be
given by a choice of a cyclic order on the vertices of triangle). We say that the
orientations of two neighboring triangles are \emph{coherent} if they are
locally both clockwise or both counterclockwise.\footnote{Note that we cannot
speak of clockwise or counter-clockwise direction in global sense on whole
$\nsurf$ since we expect to work with nonorientable surfaces. However, we
still can do this locally on a rather trivial orientable surface consisting of
the two triangles.}
\immfig{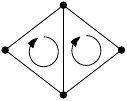}
Given a choice $\omega$ of local orientations on all triangles of $\nsurf$ we
create a graph $G_\omega$ embedded in $\nsurf$ consisting of all edges of the
triangulation for which the two neighboring triangles are not coherent. 
(Using the terminology of ~\cite{MoharThomassen}, this
corresponds to the edges, in the dual graph, of \emph{signature} $-1$.)
\immfig{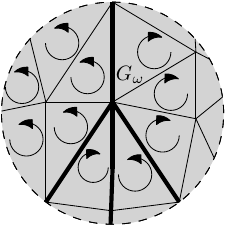}
By $\cut\nsurf\omega$ we denote the (possibly disconnected) 
surface obtained from $\nsurf$ by cutting along $G_\omega$. 
The surface $\cut\nsurf\omega$ is orientable by the choice of the cut edges. 
Therefore, in particular, if $G_\omega$ consists of a single closed curve, then this is
an orientation-enabling curve.
%
%


Given these preliminaries we can prove the following auxiliary proposition
resembling Proposition~\ref{p:oi_exists} for surfaces without boundary.

\begin{proposition}
\label{p:oi_no_holes}
Let $\nsurf$ be a nonorientable surface without boundary with a fixed triangulation with
total of $n$ vertices and edges. Then there is an orientation-enabling curve
avoiding the vertices of $\nsurf$ and meeting the edges of $\nsurf$ in at most $2n$
intersections.
\end{proposition}

\begin{proof}
First we create a certain collection of closed curves on $\nsurf$. 
Let $\omega$ be a choice of local orientations. 
For every vertex $u$
we pair edges of $G_\omega$ incident to $u$ so that the two edges in every pair
are neighbors in the cyclic order. This is possible since each edge corresponds
to a change of local orientations and when we travel around $u$ we have to
observe an even number of changes.
We shorten each edge $\varepsilon$ of $G_\omega$ and shift it a little,
obtaining a new edge $\hat \varepsilon$ that avoids the edges of the triangulation of $\nsurf$. 
We connect these shortened edges according to the chosen pairs:
\immfig{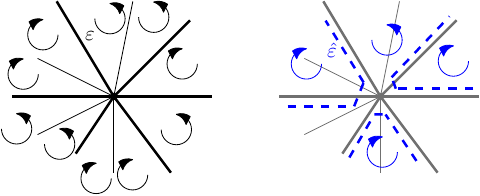}

In this way, we obtain a system of closed curves $(\gamma_1, \dots, \gamma_t)$ 
(understood as curves in $\nsurf$). 
Moreover, we can consider this system of curves 
as $G_{\eta}$ where $\eta$ is
a choice of local orientations of some suitable refinement of the
original triangulation of $\nsurf$. 

Further we observe that $G_{\eta}$ meets each edge of $\nsurf$ at most twice 
(once next to each vertex of $G_{\omega}$; we emphasize that 
by an edge of $\nsurf$ we mean an edge of the original triangulation of $\nsurf$).

If we are lucky and $t=1$, that is, $G_{\eta}$ consists of a single closed
curve, then
we deduce that this curve is the curve we seek and we are done.

If $t > 1$ we still have to modify the local orientations in order to obtain a
single closed curve. In this case we will find a further refinement of the triangulation of $\nsurf$
and a choice of local orientations $\vartheta$ such that $G_{\vartheta}$
consists of $t-1$ closed curves and $G_{\vartheta}$ still meets each edge of the original triangulation of $\nsurf$ at most twice. After repeating this step $(t-1)$ times we obtain
the required closed curve.

Let $G^*$ be the graph dual to the triangulation of $\nsurf$. That is,
the vertices of $G^*$ are the triangles of $\nsurf$ and the edges of $G^*$ are the pairs of triangles
sharing an edge. Let $\tau_1$ and $\tau_2$ be two triangles closest in $G^*$ such
that $\tau_1$ contains a part of some curve $\gamma_i$ and $\tau_2$ contains
a part of some curve $\gamma_j$ with $i \neq j$ (possibly $\tau_1 = \tau_2$).

We want to connect $\gamma_i$ and $\gamma_j$ with an arc $\delta$ that is
minimal in the following sense. First of all we assume that $\delta$ belongs
only to triangles of some preselected shortest path between $\tau_1$ and
$\tau_2$ in $G^*$. We also assume that it intersects each edge of $\nsurf$ at
most once. Finally, we can also assume that $\delta$ intersects $G_\eta$ only
in endpoints of $\delta$, for otherwise, we could shorten $\delta$ (this might
require changing the indices $i$ or $j$ if $\tau_1 = \tau_2$ and this triangle
contain other curve(s) $\gamma_k$). 
We observe that all the inner triangles on the
preselected shortest path between $\tau_1$ and $\tau_2$ are disjoint from
$G_\eta$ due to our choice of $\tau_1$ and $\tau_2$. It follows that 
if $\delta$ intersects an edge of $\nsurf$, then this edge is not intersected
by $G_\eta$.

Now we consider two arcs $\delta_1$ and $\delta_2$ parallel to $\delta$ (both
of them join $\gamma_i$ and $\gamma_j$). We join $\gamma_i$ and $\gamma_j$
into a single closed curve $\gamma'$ along $\delta_1$ and $\delta_2$:
\immfig{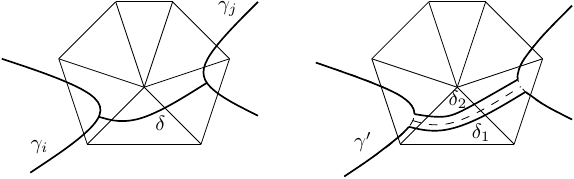}

After a suitable refinement of the triangulation we change the orientation of the narrow region
between $\delta_1$, $\delta_2$ and of the two tiny segments of $\gamma_i$ and
$\gamma_j$:
\immfig{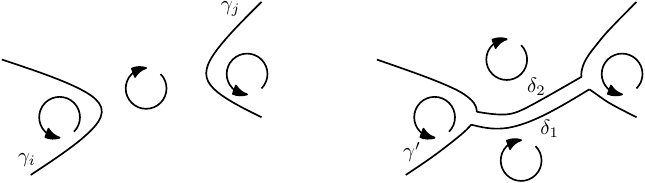}
This way we obtain the required new choice of local
orientations $\vartheta$. The corresponding graph $G_\vartheta$ consist of $\gamma'$
and all closed curves (cycles) of $G_\eta$ except $\gamma_i$ and $\gamma_j$, that is, it has
$t-1$ closed curves as required. In addition, it intersects each edge of $\nsurf$ at
most twice due to the choice of $\delta$. This finishes the proof.
\end{proof}

Now we are ready to prove Proposition~\ref{p:oi_exists}.

\begin{proof}[Proof of Proposition~\ref{p:oi_exists}.]
First we contract all boundary holes $\gamma_i$ to points $\hat \gamma_i$; 
in this
way, we obtain a surface $\hat \nsurf$. We remember the orientation of $\gamma_i$ as
one of two possible directions of how to travel around $\hat \gamma_i$ in some
neighborhood of $\hat \gamma_i$ (it does not make sense to consider whether this
direction is clockwise or counter-clockwise, since $\nsurf$ is not orientable).
We also let $\hat D = (\hat \delta_1, \dots, \hat \delta_n)$ be the
system of curves on $\hat \nsurf$ corresponding to $D$ on $\nsurf$.

Now we form a graph $G$ embedded in $\hat \nsurf$ in the following way.
The vertex set of $G$ consists of all endpoints of arcs in $\hat D$. For a
closed curve 
in $\hat D$, we pick a vertex on this curve. Each arc in $\hat D$ induces an
edge in $G$. Each closed curve in $\hat D$ induces a loop in $G$. This finishes the
construction of $G$. Note that the $\hat \gamma_i$ are situated either in the
vertices of $G$ or in the faces, but not in the interiors of the edges. 
Also note that no
two holes are contracted to the same vertex.

The graph $G$ has $O(n)$ vertices and edges. Let $G'$ be the graph from
Lemma~\ref{l:O(g+n)} containing some subdivision of $G$ and having $O(g+n)$
vertices and edges. By possibly perturbing $G'$, we can assume that the $\hat
\gamma_i$ are not in the interiors of edges of $G'$.

Using Proposition~\ref{p:oi_no_holes} we find an orientation-enabling curve $\hat
\lambda_0$ that intersects each edge of $G'$ at most twice. We would like to
decontract the holes transforming $\hat \lambda_0$ to $\lambda_0$ on $\nsurf$
getting the required curve. However, the problem is that the orientations of
curves on $\nsurf_{\langle \lambda_0 \rangle}$ may not be compatible as we require. We still
have to modify $\lambda_0$. We use an approach similar the proof of the
previous proposition.

Let $G^*$ be the dual graph to $G'$.
Let us also equip $\lambda_0$ with some
orientation. Note that $\lambda_0$ can be one-sided or two sided in $\nsurf$.
In the second case, it is important to observe that the two closed curves
originating from $\lambda_0$ on $\nsurf_{\langle \lambda_0 \rangle}$ have compatible orientations.
(Otherwise, gluing along them would mean introducing a handle, contradicting
the non-orientability of $\nsurf$.)

Let $\gamma_i$ be a hole such that the orientation of $\gamma_i$ is not
compatible with $\lambda_0$ on $\nsurf_{\langle \lambda_0 \rangle}$. Let $\tau_1$ be a
triangle containing $\hat \gamma_i$ (if $\hat \gamma_i$ is a vertex,
it may be contained in several triangles). Let
$\tau_2$ be a triangle containing a part of $\hat \lambda_0$ closest to
$\tau_1$ in $G^*$. We connect $\hat \lambda_0$ with $\hat \gamma_i$ by an arc
$\delta$ minimal in the following sense. We assume that $\delta$ uses triangles
of some prescribed shortest path between $\tau_1$ and $\tau_2$. It intersects
each edge on this path at most once. It also has no other intersection with
$\lambda_0$, for otherwise, it could be shortened.

We `pull a finger' along $\delta$ obtaining a new curve $\hat \lambda_1$:
\immfig{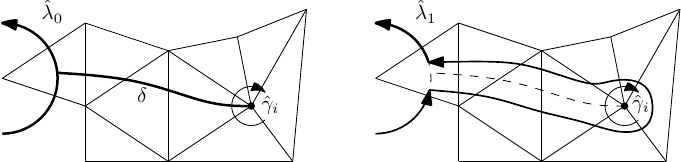}
After decontractions, we obtain that the resulting
$\lambda_1$ and $\gamma_i$ are compatible on $\nsurf_{\langle \lambda_1 \rangle}$.
The compatibility of $\lambda_1$ with respect to other boundary curves is not
affected.


The curve $\hat \lambda_1$ can have more intersections with the
edges of $G'$. However,
the new intersections appear either on edges that were not
intersected previously (at most twice), or, if $\hat \gamma_i$  is a vertex,
 on the edges incident to it.

We can apply this procedure repeatedly, obtaining $\hat \lambda_2$,
$\hat \lambda_3$, etc. After a finite number of steps we
obtain a curve $\hat \lambda_k$ such that the corresponding
$\lambda_k$ is already compatible
with all holes on $\nsurf_{\langle \lambda_k \rangle}$. This curve is our desired curve
$\lambda$, since during the procedure we have introduced at most $2|E(G')| +
\sum \deg v$ new intersections, where the sum is over all vertices $v$ of $G'$.
Thus we are still within the $O(g+n)$ bound.
\end{proof}

\ifGD
\else
\subsection*{Acknowledgement}
We would like to thank the authors of \cite{Geelen:Explicit-bounds-for-graph-minors-2013} for making a draft of their paper available to us,
and, in particular, T.~Huynh for an e-mail correspondence. We also thank an
anonymous referee for many valuable comments and in particular for a suggestion
of using local orientations in the proof of Proposition~\ref{p:oi_exists} which
replaced our original (longer) homology-based proof.
\bibliographystyle{alpha}
\bibliography{cx.bib}
\fi

\end{document}